\documentclass[10pt]{amsart}
 \usepackage[backref=page]{hyperref}
\usepackage{amsfonts,amssymb,verbatim,amsmath,amsthm,latexsym,textcomp,amscd,hyperref,epsfig}
\usepackage{tikz-cd}
\usepackage{graphics,textcomp}
\usepackage{graphicx}
\usepackage{color}
\usepackage{url}
\usepackage{enumitem}
\begin{document}

\title{$\mathbb A^1$-connected components of affine quadrics}	
\author{Chetan Balwe}	
\address{Department of Mathematical Sciences, Indian Institute of Science Education and Research Mohali, Knowledge City, Sector 81, SAS Nagar, Punjab 140306,  India}
\email{cbalwe@iisermohali.ac.in}
\author{Nidhi Gupta}	
\address{School of Mathematics, Tata Institue of Fundamental Research, Homi Bhabha Road, Colaba, Mumbai 400005, India}
\email{nidhi@math.tifr.res.in}

\begin{abstract}
For any smooth  quadratic hypersurface $X$ in $\mathbb A^n_k$, we use the iterations of the functor of naive $\mathbb{A}^1$-connected components $\mathcal{S}$ to study the field-valued sections of the sheaf of $\mathbb{A}^1$-connected components  $\pi_0^{\mathbb{A}^1}(X)$ of $X$.  We prove that for any  field $F/k$, the canonical isomorphism $\pi_0^{\mathbb{A}^1}(X)(F) \xrightarrow{\sim} \lim_{n} \mathcal{S}^n(X)(F)$ stabilizes at $n=2$, meaning that   $\pi_0^{\mathbb{A}^1}(X)(F)=\mathcal{S}^2(X)(F)$.
Furthermore, by combining this result with Morel's characterization of $\mathbb{A}^1$-connected spaces in terms of the triviality of field-valued sections of $\pi_0^{\mathbb{A}^1}$, we provide a complete characterization of $\mathbb{A}^1$-connected smooth quadratic hypersurfaces in $\mathbb{A}^n_k$.  
\end{abstract}
\maketitle
\newtheorem{theorem}{Theorem}[section]
        \newtheorem{proposition}[theorem]{Proposition}
	\newtheorem{lemma}[theorem]{Lemma}
	\newtheorem{corollary}[theorem]{Corollary}
	\theoremstyle{definition}
        \newtheorem{definition}[theorem]{Definition}
        \newtheorem{construction}[theorem]{Construction}
	\newtheorem{notation}[theorem]{Notation}
        \newtheorem*{ack}{Acknowledgement}
	\newtheorem{claim}[theorem]{Claim}
    \newtheorem*{case}{Case}
       \newtheorem{example}[theorem]{Example}
	\newtheorem{examples}[theorem]{Examples}
	\newtheorem{remark}[theorem]{Remark}

\def\~{\widetilde}
\def\-{\overline}
\def\<{\langle}
\def\>{\rangle}
\def\@{\mathcal}
\def\!{\mathbf}
\def\#{\mathbb}
\def\^{\widehat}

\newcommand{\Proj}{{\rm Proj} \,}
\newcommand{\codim}{{\rm codim}}
\newcommand{\Dim}{{\rm dim}}
\newcommand{\rank}{{\rm rank}}
\newcommand{\Spec}{{\rm Spec \,}}
\newcommand{\cha}{{\rm char} \,}
\newcommand{\fixme}[1]{\textcolor{red}{#1}}

\section{Introduction}
Let $\mathcal{H}(k)$ be the Morel–Voevodsky unstable $\mathbb{A}^1$-homotopy category over a field $k$ \cite{MV}. Similar to algebraic topology, one can associate various invariants to any simplicial sheaf $\mathcal F\in \mathcal H(k)$ to study its homotopy type. One such invariant is the sheaf of $\mathbb{A}^1$-connected components, denoted by $\pi_0^{\mathbb{A}^1}(\mathcal F)$. However, explicitly computing $\pi_0^{\mathbb{A}^1}(\mathcal F)$ is generally challenging.

Asok and Morel (see \cite{AM}) introduced the sheaf of $\mathbb{A}^1$-chain connected components, denoted in this article by $\mathcal{S}(\mathcal{F})$. Intuitively, $\mathcal{S}(\mathcal{F})$ (see Definition \ref{naive definiton}) captures connectedness through homotopies parametrized by the affine line. There exists a canonical epimorphism of sheaves:  
\[
\mathcal{S}(\mathcal{F}) \to \pi_0^{\mathbb{A}^1}(\mathcal{F}).
\]  
This epimorphism need not be an isomorphism in general, even when $\mathcal{F}$ is represented by a smooth proper scheme \cite{BHS, BSruled}.  

To further understand $\pi_0^{\mathbb{A}^1}(\mathcal{F})$, one considers the universal $\mathbb{A}^1$ invariant quotient of $\mathcal F$, denoted $\mathcal{L}(\mathcal{F})$, which is obtained as the colimit of the sequence $\mathcal{S}^n(\mathcal{F})$ for $n \in \mathbb{N}$. The canonical map  
\[
\pi_0^{\mathbb{A}^1}(\mathcal{F}) \to \mathcal{L}(\mathcal{F})=\lim_{\to n} \mathcal{S}^n(\mathcal{F})
\]  
is an isomorphism if and only if $\pi_0^{\mathbb{A}^1}(\mathcal{F})$ is $\mathbb{A}^1$-invariant \cite{BHS}. Morel conjectured that $\pi_0^{\mathbb{A}^1}(\mathcal{F})$ is always $\mathbb{A}^1$-invariant \cite{Morelbook}. Ayoub gave a counter-example to this conjecture (see \cite{Aconjecture}) but it remains open for representable sheaves.

Morel's conjecture holds when $\mathcal{F}$ is represented by a scheme of dimension $\leq 1$, an $\mathbb{A}^1$-rigid scheme, or an $\mathbb{A}^1$-connected scheme. Additionally, it has been proved for motivic $H$-groups, homogeneous spaces of motivic $H$-groups \cite{CU}, and smooth toric varieties \cite{Wendt}. Using the universal $\mathbb{A}^1$-invariant quotient, Morel's conjecture has been verified for smooth projective surfaces \cite{BHS} and certain smooth projective threefolds \cite{Pawar}.  

If $\mathcal{F}$ is a sheaf of sets and $K/k$ is a finitely generated separable extension, the natural map     
\[
\pi_0^{\mathbb{A}^1}(\mathcal{F})(K) \to \mathcal{L}(\mathcal{F})(K)
\]  
is a bijection (see \cite[Theorem 1.1]{BRS}. Thus, to compute the field-valued points of $\pi_0^{\mathbb{A}^1}(\mathcal{F})$, we must iterate the functor $\mathcal{S}$ infinitely many times. However, if $\mathcal{F}$ is represented by a proper scheme $X$, the infinitely many iterations are inessential because $\mathcal{L}(X)(K)=\mathcal{S}(X)(K)$ (see \cite[Proposition 2.4.3]{AM} or \cite[Theorem 2]{BHS}). On the other hand, for every $n \in \mathbb{N}$, there exists a non-proper complex variety $X_n$ for which $\mathcal{S}^n(X_n)(\mathbb{C}) \neq \mathcal{S}^{n+1}(X_n)(\mathbb{C})$, while $\mathcal{L}(X_n)(F)=\mathcal{S}^{n+1}(X_n)(F)$ for every field extension $F/\mathbb C$ (see \cite{gupta}). Thus, there is no finite, uniform bound on the number of iterations of $\mathcal{S}$ that are required in order to compute the field-valued points of the sheaf of $\mathbb{A}^1$-connected components of a variety. It is natural to ask the following question: For a given variety $X$, does there exist $n_X \in \mathbb{N}$ such that
\[
\pi_0^{\mathbb{A}^1}(X)(K)=\mathcal{S}^{n_X}(X)(K) \quad \text{for all } K/k?
\]  

The aim of this paper is to compute the field-valued sections of the sheaf $\pi_0^{\mathbb{A}^1}(X)$ for any smooth affine quadratic hypersurface $X$ in $\mathbb{A}^n_k$ and to show that the sequence $\mathcal{S}^n(X)(K)$ stabilizes at $n=2$ for all $K/k$. (The case of non-smooth affine quadratic hypersurfaces is quite trivial, as can be seen in Remark \ref{remark non-smooth case}.)

If $\psi$ is a quadratic form in $n$ variables $x_1, \ldots, x_n$, we will denote by $Q^{\psi}$ the quadratic hypersurface in $\mathbb{A}^n_k$  defined by the equation $\psi(x_1, \ldots, x_n)-1=0$. We will see in Theorem \ref{theorem standard form} that it is sufficient to consider quadratic hypersurfaces of this type. Our first main result is the following:

\begin{theorem}\label{main theorem for Q}
    Let $k$ be a field of characteristic 0 and let $F/k$ be a finitely generated field extension. Let $\psi$ be a regular quadratic form in $n \geq 3$ variables, and let $Q^\psi$ be the smooth hypersurface in $\mathbb{A}^n_k$ defined by the equation $\psi(x_1, \dots, x_n)-1=0$.  
    \begin{enumerate}
        \item If $\psi$ is isotropic, then $\psi \sim \< 1 \> \perp -\phi$ for some quadratic form $\phi$ in $n-1$ variables, and  
        \[
        \pi_0^{\mathbb{A}^1}(Q^{\psi})(F)=\mathcal{S}(Q^{\psi})(F)=\frac{F^*}{\< D(\phi_F) \>}.
        \]  
        \item If $\psi$ is anisotropic, then $\mathcal S(Q^{\psi})(k)\neq \mathcal S^2(Q^{\psi})(k)$ and  
        \[
        \pi_0^{\mathbb{A}^1}(Q^{\psi})(F)=\mathcal{S}^2(Q^{\psi})(F).
        \]
    \end{enumerate}
\end{theorem}

Analogous results for algebraic groups have been established in \cite{BSgroups, BSgroups2}. By a result of Morel, a sheaf $\mathcal{F}$ is $\mathbb{A}^1$-connected if and only if $\pi_0^{\mathbb{A}^1}(\mathcal{F})(K)$ is trivial for any field $K/k$. Combining this with Theorem \ref{main theorem for Q}, we provide a complete characterization of the $\mathbb{A}^1$-connected quadratic hypersurfaces in $\mathbb{A}^n_k$ in terms of classical invariants of quadratic forms. 

For a quadratic form $\theta$ over a vector space $V$, let $i_0(\theta)$ denote the Witt index of $\theta$, that is, the dimension of the maximal isotropic subspace of $V$. Let $i_1(\theta)=i_0(\theta_{k(\theta)})$, where $k(\theta)$ is the function field of the projective quadric corresponding to $\theta$.  Then, we have proved the following result:

\begin{theorem}\label{a1connected quadric}
    Let $\psi$ be a quadratic form in $n \geq 3$ variables and let $\kappa$ denote the  quadratic $\psi \perp \< -1\>$. Then $Q^{\psi}$ is $\mathbb{A}^1$-connected if and only if one of the following conditions holds:  
    \begin{enumerate}
        \item $i_0(\kappa) \geq 2$,
        \item $i_0(\kappa)=1$, $i_0(\psi)=0$, and $i_1(\psi) \geq 2$.
    \end{enumerate}
\end{theorem}  

Consequently, if $k$ is algebraically closed (or quadratically closed), any smooth affine quadratic hypersurface in $\mathbb{A}^n_k$ for $n\geq 2$ is either isomorphic to $\mathbb G_m\times \mathbb A^{n-2}_k$ or is  $\mathbb{A}^1$-connected.

\section{Preliminaries}

\subsection{Preliminaries on $\mathbb A^1$-connected components}

We fix a base field $k$. Let $Sm/k$ denote the Grothendieck site of
finite-type smooth schemes on $k$ equipped with the Nisnevich topology.

\begin{definition}
    Let $\@F$ be a sheaf of sets on $Sm/k$. The sheaf of \textit{$\mathbb A^1$-connected components} of $\mathcal F$, denoted by $\mathbb \pi_0^{\mathbb A^1}(\mathcal F)$ is defined as the Nisnevich sheafification of the presheaf 
    \[U\in Sm/k\mapsto Hom_{\mathcal H(k)}(U,\mathcal F).\]
\end{definition}

\begin{definition}
    A sheaf of sets $\mathcal{F}$ is said to be $\mathbb{A}^1$-connected if the canonical morphism $\mathcal{F} \to \Spec k$ induces the isomorphism of sheaves $\pi_0^{\mathbb{A}^1}(\mathcal{F}) \xrightarrow{\sim} \Spec k$.
\end{definition}

The following result allows us to determine whether $\mathcal F$ is $\mathbb A^1$-connected by just examining the field-valued sections of $\pi_0^{\mathbb A^1}(\mathcal F)$. 

\begin{lemma}\cite[Lemma 6.1.3]{MF}\label{stable}
    A sheaf of sets $\mathcal{F}$ on $Sm/k$ is $\mathbb{A}^1$-connected if $\pi_0^{\mathbb{A}^1}(\mathcal{F})(K)=*$ for every finitely generated separable extension $K$ of $k$.
\end{lemma}

\begin{notation}
For any smooth scheme $U$ over $k$, $t\in k$, and any $H\in \mathcal{F}( \mathbb A^1_k\times U)$, define $H(t):= H\circ s_t$, where $s_t$ is the morphism $U\rightarrow \mathbb A^1_k\times U$ given by $u\mapsto (t,u)$. 
\end{notation}

\begin{definition}
    Let $\mathcal F$ be a sheaf of sets. Let $x_0,x_1\in\mathcal F(U)$.
    \begin{enumerate}
        \item $h\in \mathcal F(\mathbb A^1_k\times U)$ is called an \textit{$\mathbb A^1$-homotopy of $U$} connecting $x_0$ and $x_1$ if $h(0)=x_0$ and $h(1)=x_1$.
        \item $x_0$ and $x_1$ are  said to be \textit{$\mathbb A^1$-homotopic} if there exist an $\mathbb A^1$-homotopy $h\in \mathcal F(\mathbb A^1_k\times U)$ connecting $x_0$ and $x_1$. We say that $x_0$ and $x_1$ are \textit{$\mathbb A^1$-chain homotopic} if there exists $\mathbb A^1$-homotopies $h_i\in \mathcal F(\mathbb A^1_k\times U)$ and $z_i\in \mathcal F(U)$ for $i=1,\dots, r$ such that $h_i$ connects $z_{i-1}$ and $z_i$ where $z_0=x_0$ and $z_r=x_1$.
     \end{enumerate}
\end{definition}

\begin{definition}\label{naive definiton}
The sheaf of \textit{naive $\mathbb A^1$-connected components} of $\mathcal F$, denoted by $\mathcal S(\mathcal F)$, is defined as the Nisnevich sheafification of the presheaf $\mathcal S^{pre}(\mathcal F)$, 
\[
\mathcal S^{pre}(\mathcal F)(U):= \frac{\mathcal F(U)}{\sim},
\]
where $\sim$ is the equivalence relation of $\mathbb A^1$-chain homotopy. 
\end{definition}

The sheaf $\mathcal S(\mathcal F)$ is also called  the \emph{sheaf of $\mathbb{A}^1$-chain connected components of $\@F$} (see \cite[Definition 2.2.4]{AM}. It is easy to see that $\mathcal{S}(\mathcal{F})=\pi_0(Sing_*^{\mathbb{A}^1}(\mathcal{F}))$ where $Sing_*^{\mathbb{A}^1}(\mathcal F)$ is the Suslin–Voevodsky singular construction of $\mathcal F$ (see \cite[p. 87]{MV}). 

We will say that a sheaf of sets $\@F$ is \emph{$\#A^1$-chain connected} if $\mathcal{S}(\mathcal{F})(K)=\ast$ for any finitely generated separable field extension $K/k$ (see \cite[Definition 2.2.2]{AM}). Note that this condition does not imply that $\mathcal{S}(\mathcal{F})=\ast$ but, due to Lemma \ref{stable},  it does imply that $\pi_0^{\mathbb{A}^1}(\mathcal{F})=\ast$. 

For a sheaf $\mathcal{F}$ and positive integer $n$, let $\mathcal{S}^n(\mathcal{F})$ be the sheaf obtained from $\mathcal F$ by applying $n$ iterations of the functor $\mathcal S$ . We have the following sequence of epimorphisms:
\[
\mathcal{F} \to \mathcal{S}(\mathcal{F}) \to \mathcal{S}^2(\mathcal{F}) \to \dots.
\]
Taking the direct limit of this sequence, we arrive at the \textit{universal $\mathbb{A}^1$-invariant quotient} $\mathcal{L}(\mathcal{F})$:
\[
\mathcal{L}(\mathcal{F}) := \lim_{\to n} \mathcal{S}^n(\mathcal{F}).
\]

For a field $K/k$, \cite[Theorem 1.1]{BRS} states that any two $K$ points of $\mathcal F$ have the same image in $\pi_0^{\mathbb A^1}(\mathcal F)(K)$ if and only if they have the same image in $\mathcal S^{n}(\mathcal F)(K)$ for some $n\in \mathbb N$.
Since we are interested in computing $\pi_0^{\mathbb A^1}(X)(K)$ for scheme $X$, it will be useful to have an explicit description of $\mathbb A^1$-homotopies of a field $K/k$ in $\mathcal S^n\mathcal ( X)$ via its lift to $X$. For this, we will use the notion of $n$-ghost homotopies. The following definition is a slightly modified reformulation of the one appearing in \cite[Definition 3.2]{BHS} and \cite[Definition 2.7]{BSruled}.

 \begin{definition}\label{data of 1-ghost homotopy}
Let $\mathcal F$ be a sheaf of sets, and $K/k$ be a finitely generated field extension. Let $n>0$. Let $x_0,x_1\in \mathcal F(K)$.
\begin{enumerate}  
    \item An \textit{$n$-ghost homotopy} $\mathcal{H}$ of $K$ connecting $x_0$ and $x_1$ in $\mathcal{F}$ consists of the data  
    \[
    \left( (U, V),\, h_U, h_V,\, y_0, y_1 \right),
    \]
    where:
    \begin{enumerate}  
        \item $(U \hookrightarrow \mathbb{A}^1_K,\, V \xrightarrow{p} \mathbb{A}^1_K)$ is an elementary Nisnevich covering of $\mathbb{A}^1_K$,  
        \item The maps $h_U: U \to \mathcal{F}$ and $h_V: V \to \mathcal{F}$ satisfy  $pr_U^*(h_U)=pr_V^*(h_V)$ in  $\mathcal{S}^{n}(\mathcal{F})(U \times_{\mathbb{A}^1_K} V)$,
        where $pr_U$ and $pr_V$ are the projection maps from $U \times_{\mathbb{A}^1_K} V$ to $U$ and $V$, respectively.  
        \item For $i=0,1$, $(h_U \coprod h_V)(y_i)=x_i.$ 
    \end{enumerate}  
    \item The elements $x_0$ and $x_1$ are called \textit{$n$-ghost homotopic} in $\mathcal{F}$ if there exists an $n$-ghost homotopy connecting them.  
    They are called \textit{$n$-ghost chain homotopic} if there exist $n$-ghost homotopies $\mathcal{H}_i$ and elements $z_i \in \mathcal{F}(K)$ for $i=1, \dots, r$ such that $\mathcal{H}_i$ connects $z_{i-1}$ and $z_i$, where $z_0=x_0$ and $z_r=x_1$.  
\end{enumerate}  
\end{definition}  

Since $\dim \mathbb{A}^1_K=1$, any Nisnevich covering of $\mathbb{A}^1_K$ can be refined to an elementary Nisnevich covering. Therefore, any map $\mathbb{A}^1_K \to \mathcal{S}^n(\mathcal{F})$ is defined by data as in Definition \ref{data of 1-ghost homotopy}(1). Conversely, such data induce a map $\mathbb{A}^1_K \to \mathcal{S}^n(\mathcal{F})$.  

\begin{notation}\label{sq}  
For any sheaf of sets $\mathcal{F}$, a scheme $U$, and an element $x \in \mathcal{F}(U)$, we denote the image of $x$ in $\mathcal{S}^j(\mathcal{F})(U)$ by $[x]_j$.  
\end{notation}

The following lemma is a restatement of a result from \cite{BSruled} and follows directly from the definitions.  

\begin{lemma}\cite[Lemma 2.8]{BSruled}
    Let $\mathcal F$ be a sheaf of sets, and $K/k$ be a finitely generated field extension. Let $x_0,x_1\in \mathcal F(K)$. Then $[x_0]_n=[x_1]_n$ if and only if $x_0$ and $x_1$  are $n$-ghost chain homotopic in $\mathcal F$.
\end{lemma}

Building on \cite[Lemma 2.1]{BRS}, the next lemma simplifies the data of an $n$-ghost homotopy of a field $K$.  
\begin{lemma}\label{lemma on lines of brs}
   Let $\mathcal{F} \in Shv(Sm/k)_{Nis}$. Suppose we have the following data:
   \begin{enumerate}
       \item $(U,V \xrightarrow{p} \mathbb{A}^1_k)$ is an elementary Nisnevich cover of $\mathbb{A}^1_k$ with $W=U\times_{\mathbb{A}^1_k}V$.
       \item $h_U \in \mathcal{F}(U)$, $h_V \in \mathcal{F}(V)$, and there exists a dense open subscheme $W' \subset W$ such that $pr_U^*|_{W'}(h_U)$ and $pr_V^*|_{W'}(h_V)$ have the same image in $\mathcal{S}(\mathcal{F})(W')$, where $pr_U$ and $pr_V$ are the projection maps from $W$ to $U$ and $V$, respectively.
   \end{enumerate}
   Let $h=h_U \coprod h_V$. Then $[h(P)]_2=[h(P')]_2$ for any $P, P' \in (U \coprod V)(k)$.
\end{lemma}

\begin{proof}
We claim that there exists an open subscheme $V' \subseteq V$ such that the pair $(U, V' \xrightarrow{p|_{V'}} \mathbb{A}^1_k)$ is an elementary Nisnevich cover and $W'=U \times_{\mathbb{A}^1_k} V$. 

Indeed, let $Z=\mathbb{A}^1_k \setminus U$. Since $(U, V)$ is an elementary Nisnevich cover of $\mathbb{A}^1_k$, the restriction $p^{-1}(Z) \to Z$ is an isomorphism. Define $V'=W' \cup p^{-1}(Z)$. Since $\mathbb{A}^1_k$ is 1-dimensional, $V'$ is an open subscheme of $V$. Clearly, $V'$ is the required scheme. Let $h_{V'}=h_V|_{V'}$. Then, 
\[\left((U, V'\xrightarrow{p|_{V'}}\mathbb A^1_k), W', h_U, h_{V'}\right)\]
forms a data for $1$-ghost homotopy $\mathcal H$.

Therefore, for any $P, P' \in (U \coprod V')(k)$, we have $[h(P)]_2=[h(P')]_2$ via $\mathcal H$. 

Now, similarly for any $Q\in (V\setminus V')(k)$,  there exists a Nisnevich cover  $(U'\subset U,V''\subset V)$ such that $W'=U' \times_{\mathbb{A}^1_k} V''$ and $P\in V''(k)$. Indeed, let $V''=V'\cup\{Q\}$ and $U'=U\setminus\{p(Q)\}$. Then, 
\[\left((U', V''\xrightarrow{p}\mathbb A^1_k), W',  h_{U'}, h_{V''}\right)\]
forms the data for a $1$-ghost homotopy $\mathcal H'$. Thus,  $[h(Q)]_2=[h(P)]_2$ for any $P \in (U \coprod V')(k)$ via a chain of $\mathcal{H}$ and $\mathcal{H}'$. This completes the proof.
\end{proof}

We present a straightforward consequence of the above lemma, which will be used in the proof of Theorem \ref{main theorem for Q}.  

\begin{lemma}\label{lemma on stabilisation}
    Let $\mathcal{F} \in Shv(Sm/k)$. Suppose that there exists $n \in \mathbb{N}$ such that $\mathcal{S}^n(\mathcal{F})(K)=\mathcal{S}^{n+1}(\mathcal{F})(K)$ for all $K/k$. Then
    \[\mathcal{L}(\mathcal{F})(K)=\mathcal{S}^n(\mathcal{F})(K) \text{ for all } K/k.\]
\end{lemma}

\begin{proof}
    We prove that $\mathcal{S}^n(\mathcal{F})(K)=\mathcal{S}^{n+i}(\mathcal{F})(K)$ for all $K/k$ and for all $i \geq 0$ by induction on $i$. 

    The base case $i=1$ follows from the assumption. Suppose that the statement holds for some $i=m>1$, i.e., $\mathcal{S}^n(\mathcal{F})(K)=\mathcal{S}^{n+m}(\mathcal{F})(K)$ for all $K/k$. We show that it holds for $i=m+1$.

    Suppose $[Q_1]_{n+m+1}= [Q_2]_{n+m+1} \in\mathcal S^{n+m+1}(F)(K)$ by an $(n+m)$-ghost homotopy $\mathcal H: \mathbb{A}^1_K \to \mathcal{S}^{n+m}(\mathcal{F})$ which is given given by the data
    \[
    \left((U, V \xrightarrow{p} \mathbb{A}^1_K), W, H_U, H_V, P_1, P_2\right)
    \]
    such that $(H_U\sqcup H_V)(P_i)=Q_i$ and $[pr_U^*(H_U)]_{n+m}=[pr_V^*(H_V)]_{n+m}$ in $\mathcal{S}(\mathcal{F})(W)$.

    If $W=\sqcup_{j=1}^{r} W_j$ is the decomposition of $W$ into irreducible components, we have for each $j$,
    \[
    [pr_U^*|_{k(W_j)}(H_U)]_{n+m}=[pr_V^*|_{k(W_j)}(H_U)]_{n+m}.
    \]
    By the induction hypothesis, for each $j$,
    \[
    [pr_U^*|_{k(W_j)}(H_U)]_{n}=[pr_V^*|_{k(W_j)}(H_U)]_{n}.
    \]
    Hence, there exists a dense open subscheme $W'$ of $W$ such that
    \[
    [pr_U^*|_{W'}(H_U)]_{n}=[pr_V^*|_{W'}(H_U)]_{n}.
    \]
    By Lemma \ref{lemma on lines of brs}, we conclude that $[Q_1]_n=[Q_2]_n$, completing the proof.
\end{proof}

The first assertion in the following theorem is due to Asok and Morel. Although they only state it for proper varieties, their proof also holds in the case of smooth non-proper varieties. The second assertion is an easy consequence of their proof. 

\begin{theorem}\cite[Proposition 2.3.8]{AM} \label{AM}  
Let $\pi: X \to Y$ be the blow-up of a smooth scheme along a smooth subscheme $Z$ of codimension $r+1$. Then, for any finitely generated separable field extension $K/k$, the map  
\[
\mathcal{S}(T)(K) \xrightarrow{\sim} \mathcal{S}(S)(K)
\]  
is a bijection. Moreover, suppose that $P, Q \in Y(K)$ lie on a smooth curve isomorphic to $\mathbb{A}^1_K$. If $\pi(P')=P$ and $\pi(Q')=Q$, then $P'$ and $Q'$ are identified in $\mathcal{S}(X)(K)$ via a chain of smooth rational curves, isomorphic to $\mathbb{A}^1_K$, in $X$.  
\end{theorem}  

\subsection{Preliminaries on Quadratic Forms}  

In this subsection, we recall some classical results from the theory of quadratic forms to ensure that this exposition is self-contained.  

\begin{definition}  
    Let $\varphi$ be a quadratic form on $V$ over $k$, and let $d \in k^*$. We say that $\varphi$ represents $d$ if there exists $v \in V$ such that $\varphi(v)=d$. We denote by $D(\varphi)$ the set of values in $k^*$ represented by $\varphi$, i.e.,
    \[
    D(\varphi) := \{d \in k^* \mid \exists v \in V \text{ such that } \varphi(v)=d\}.
    \]  
\end{definition}
\begin{notation}  
    For a quadratic form $\varphi$ on a vector space $V$, let $\< D(\varphi) \>$ denote the subgroup of $k^*$ generated by the set $D(\varphi)$. For any field extension $K/k$, let $\varphi_K$ denote the quadratic form on $V \otimes_k K$ defined by $\varphi \otimes \operatorname{id}$.  
\end{notation}  

\begin{notation}  
    For any closed point $P$ of a smooth curve $C$ over $k$, let $v_P$ denote the discrete valuation $v_P: k(C) \to \mathbb{Z}$ corresponding to the discrete valuation ring $\mathcal{O}_{C,P}$.  
\end{notation}  

The following theorem is a restatement of a theorem from \cite[Theorem 18.3]{EKM} that provides a criterion for a rational function $f \in k(t)$ to belong to the group $\< D(\varphi_{k(t)}) \>$. Although the original theorem was stated for polynomials, since we have $f \< D(\varphi_{k(t)}) \>=f^{-1}  \< D(\varphi_{k(t)}) \>$, it can be easily extended to rational functions.

\begin{proposition}[Quadratic Value Theorem]\label{QVT}
     Let $f\in k(t)$. Then, the following are equivalent.
    \begin{enumerate}
        \item There exists $a\in k^*$ such that $f(t)\in a \<D(\varphi_{k(t)})\>$.
        \item  If $P$ is a closed point of $\mathbb A^1_k$ such that $v_P(f)$ is odd, then $\varphi_{\kappa(P)}$ is isotropic.
    \end{enumerate}
\end{proposition}

The following lemma shows that any quadratic form over $k$ that is isotropic over the function field of a DVR must also be isotropic over its residue field. This result is well known (see \cite{CT, Knebusch}), and the proof is simple, but we include it for completeness.
\begin{lemma}\label{lemma on isotropic forms over DVR}
    Let $(R,m,\kappa)$ be a DVR containing $k$ with function field $K$. Let $\varphi$ be a quadratic form over $k$ given by $\< a_1,a_2,\dots,a_n\>$ such that $\varphi_{K}$ is isotropic. Then $\varphi_{\kappa}$ is isotropic.
\end{lemma}
\begin{proof}
     Since $\varphi_{K}$ is isotropic, by clearing denominators, there exists $x_i\in  R $ such that \[a_1x_1^2+\dots+a_nx_n^2=0.\] 
    Each $x_i$ can be written as $x_i=t^{m_i} u_i$, where $t$ is a uniformizer ($\< t \>=m$), and $u_i$ is a unit in $R$. We may assume that not all $x_i$ belong to $m$, meaning that at least one $\bar{x}_i \neq 0$ in $\kappa$. Reducing modulo $m$, we obtain  
    \[
    a_1 \bar{x}_1^2 + \dots + a_n \bar{x}_n^2=0.
    \]  
    Thus, $\varphi_{\kappa}$ is isotropic, as required.
    \end{proof}
In particular, if we take $R$ to be the local ring of a closed point on a smooth curve $C$, the lemma implies that if $\varphi_{k(C)}$ is isotropic, then $\varphi_{k(P)}$ must be isotropic for any closed point $P$ of $C$.  

The following easy lemma follows directly from the definitions and will be useful in later sections.

\begin{lemma} \label{lemma for specialisation}
    Let $(R, m, \kappa)$ be a DVR containing $k$. Suppose $\varphi$ is a quadratic form over $k$ given by $\< a_1, a_2, \dots, a_n \>$ such that $\varphi_{\kappa}$ is anisotropic.  
    \begin{enumerate}
        \item If $y_i \in K$, then  
        \[
        v(\varphi(y_1, \dots, y_n))=\min\{ v(y_i^2) \mid 1 \leq i \leq n \}
        \]  
        and hence is even.
        
        \item Let $g \in K$ such that $g \in \< D(\varphi_{K}) \>$. If $v(g) \geq 0$, then $\bar{g}=0$ or $\bar{g} \in \< D(\varphi_{\kappa}) \>$.
    \end{enumerate}  
\end{lemma}

\begin{proof}
    For (1), without loss of generality assume $v(y_1)$ is minimum. Then, for all $i > 1$, we have $y_i / y_1 \in R$, and  
    \[
    v\left(\sum_{i=1}^{n} a_i y_i^2 \right)=v(y_1^2) + v\left( a_1 + \sum_{i=2}^{n} a_i \frac{y_i^2}{y_1^2} \right).
    \]
    Since $\varphi_{\kappa}$ is anisotropic, we have  
    \[
    a_1 + \sum_{i=2}^{n} a_i \left(\overline{\frac{y_i}{y_1}}\right)^2 \neq 0 \text{ in } \kappa.
    \]  
    Hence,  
    \[
    v\left(a_1 + \sum_{i=2}^{n} a_i \frac{y_i^2}{y_1^2}\right)=0,
    \]
    which implies that  
    \[
    v\left(\sum_{i=1}^{n} a_i y_i^2\right)=v(y_1^2),
    \]  
    completing the proof of (1).

    For (2), let $g=h_1 h_2 \dots h_n$, where $h_i \in D(\varphi_K)$. By (1), we can write  
    \[
    h_i=u_i t^{2j_i}
    \]  
    for some integer $j_i$, where $u_i \in R^* \cap D(\varphi_K)$ and $\< t \>=m$. Therefore, if $v(g)=0$, we obtain  
    \[
    \bar{g}=\bar{u}_1 \dots \bar{u}_n \in \< D(\varphi_{\kappa}) \>,
    \]  
    and if $v(g) > 0$, we have $\bar{g}=0$.  
\end{proof}

\begin{definition}
    Let $\phi$ be a quadratic form defined on a vector space $V$ over a field $k$.  
    The \textit{Witt index} $i_0(\phi)$ of $\phi$ is the dimension of a maximal totally isotropic subspace of $V$.  
    If $\phi$ is anisotropic (that is, $i_0(\phi)=0$), the \textit{first Witt index} $i_1(\phi)$ of $\phi$ is defined by
    \[
        i_1(\phi) := i_0(\phi_{k(\phi)}),
    \]
    where $k(\phi)$ denotes the function field of the projective quadric corresponding to $\phi = 0$.
\end{definition}

We end this subsection by stating two results on the first Witt index, which will be used to characterize $\mathbb{A}^1$-connected affine quadrics in Theorem \ref{a1connected quadric}. The first result provides an alternative definition of the first Witt index.

\begin{proposition}\cite[Corollary 25.3]{EKM}\label{alternative defn of i1}  
   Let $\psi$ be a non-degenerate anisotropic quadratic form over $F$ of dimension at least 2. Then  
    \[
    i_1(\psi)=\min \{ i_0(\psi_F) \mid F/k \text{ is a field extension with } \psi_F \text{ isotropic} \}.
    \]  
\end{proposition}  

\begin{proposition}\cite[Corollary 4.9]{Vishik}\label{Witt index of subform}  
    Let $\varphi$ be a non-degenerate anisotropic quadratic form over $F$ of dimension at least 2.  
    Let $\psi$ be a non-degenerate subform of $\varphi$. If $\operatorname{codim}_{\varphi}(\psi) \geq i_1(\varphi)$, then the form $\psi_{F(\varphi)}$ is anisotropic.  
\end{proposition}

\section{Classification of affine quadrics}

We assume that $k$ is a field  of characteristic $0$.
\begin{notation}\label{definition of quadrics} Let $\psi$ be a regular quadratic form over $k$ in $n$-variables. We will denote by $Q^{\psi}$, the hypersurface in $\mathbb A^n_k$ which is the locus of zeroes of the polynomial $\psi(x_1, \ldots, x_n)-1$, where $x_1, \ldots, x_n$ are the coordinate variables on $\mathbb{A}^n_k$. 
\end{notation}

Note that $Q^{\psi}$ is a smooth variety. 

\begin{theorem}
\label{theorem standard form}
    Any smooth quadratic hypersurface in $\mathbb A^n_k$ is isomorphic to $\mathbb A^{n-1}_k$ or $Q^\psi\times \mathbb A^{n-m}_k$ for some regular quadratic form $\psi$ in $m$ variables, where $1\leq m\leq n$. 
\end{theorem}
\begin{proof}
    First, we show the result for $n=1$. Let $Q$ be the closed subvariety of $\mathbb A^1_k$  cut out by $ax_1^2+bx_1+c=0$. Then, one can check that
   \[ax_1^2+bx_1+c=\frac{1}{4a}\left ((2ax_1+b)^2-(b^2-4ac)\right ).\]
  Since $Q$ is smooth, we have $b^2-4ac\neq 0$, and hence $Q\cong Q^{\psi}$ where $\psi=\< 1/(4a(b^2-4ac))\>$.
   
   For $n> 1$, let $Q$ be the hypersurface in $\mathbb A^n_k$ cut out by
   \[\sum_{1\leq i\leq j\leq n} a_{i,j}x_ix_j +\sum_{i=1}^{n}b_ix_i +c.\]
   Since any quadratic form in $n$ -variables is equivalent to a diagonal form, there exists a linear automorphism of $\mathbb A^n_k$ such that $x_i\mapsto X_i:=\sum_{j=1}^n c_{i,j}x_i$ and 
\[\sum_{1\leq i\leq j\leq n} a_{i,j}x_ix_j+\sum_{i=1}^{n}b_ix_i+c=\sum_{i=1}^{n}A_iX_i^2+\sum_{i=1}^{n}B_iX_i+c.\]
for some $A_i,B_i\in k$.
Now, if $A_i$ and $B_i$ are both non-zero for some fixed $i$, then by the Case $n=1$,
\[A_iX_i^2+B_iX_i= \frac{1}{4A_i}\left ((2A_iX_i+B_i)^2-B_i^2\right ).\]
Hence, there exists a linear automorphism of $\mathbb A^n_k$ such that $Q$ is isomorphic to $Q'$ where $Q'$ is defined by
\[\sum_{i=1}^{m}a_iy_i^2+d=0 \quad \text{or }\sum_{i=1}^{r}a_iy_i^2+\sum_{i=r+1}^{m}a_iy_i+d=0\]
where $a_i\neq 0$ for all $i$ and $ m\leq n$.

In the first case, since $Q$ is smooth we have $d\neq 0$, and $Q'$ is isomorphic to  $Q^{\psi}\times \mathbb A^{n-m}$ where $\psi=(-1/d)\< a_1,\dots,a_m\>$. In the second case, there are linear terms, and hence $Q'$ is isomorphic to the hypersurface cut out by
\[
f(y_1, \dots, y_{m-1}) + a_m y_m=0,
\]
which implies that $Q' \cong \Spec k[y_1, \dots, \widehat{y_m}, \dots, y_n]$.
\end{proof}

\begin{remark}
\label{remark non-smooth case}
    From the previous proof, we can deduce that if $X$ is a non-smooth quadratic hypersurface in $\mathbb{A}^n_k$, then $X$ is isomorphic to the affine variety
    \[
    \Spec k[x_1, \dots, x_n] /\< \psi(x_1, \dots, x_m) \>
    \]
    for some quadratic form $\psi$ in $m$ variables. It is straightforward to check that any map $U \rightarrow X$ can be made $\mathbb{A}^1$ homotopic to the constant map $U \rightarrow \Spec k \xrightarrow{(0, \dots, 0)} X.$ Therefore, $\mathcal{S}(X)=*$ trivially in this case.
\end{remark}

The previous theorem shows that it is sufficient to examine quadratic surfaces $Q^{\psi}$. 
We will first consider the easy cases when dimension of $Q^{\psi}$ is less than $2$.

\begin{theorem}
    Let $\psi$ be a regular quadratic form in $n\leq 2$ variables. Then for every finitely generated field extension $K/k$, \[\pi_0^{\mathbb A^1}(Q^{\psi})(K)=\mathcal S(Q^{\psi})(K)=Q^{\psi}(K).\]
\end{theorem}
\begin{proof}
    First, we will consider $n=1$. Then $\psi$ is of the form $\<a\>$ for some $a \in k^*$ and $Q^{\<a\>}$ is a closed subscheme of $\mathbb A^1_k$ of dimension $0$ cut out by
    \[ ax_1^2-1.\]
    If $a$ is a square in $k$, then $Q^{\<a\>}\cong \Spec  k\coprod \Spec  k$ and if $a$ is not a square in $k$, then $Q^{\<a\>}\cong \Spec  k(\sqrt a)$.  Hence, the result is clear in this case.

Now, suppose $n=2$ and let $X:=Q^{\psi}$ where $\psi=\<a_1,a_2\>$. Then $X$ is given by
    \[a_1x_1^2+a_2x_2^2-1,\]
    where $a_1,a_2\neq 0$.
    If $X(K)\neq \emptyset$, then $X_K\cong \mathbb A^1_K\setminus Q^{\<a_1/a_2\>}$, hence $X_K$ is a strictly open subscheme of $\mathbb A^1_K$. Therefore, $\pi_0^{\mathbb A^1}(X)(K)=X(K)$ for all $K/k$ as required.
\end{proof}

\section{$\mathbb{A}^1$-connected components of quadric hypersurfaces}
\label{section connected components}
In this section, we prove Theorem \ref{main theorem for Q}. We begin with a brief discussion of the setup.

Let $\psi$ be a regular quadratic form in $n \geq 3$ variables over a field $k$ of characteristic $0$. For any finitely generated field extension $F/k$, we have the sequence of sets $\{\mathcal{S}^n(Q^{\psi})(F)\}_{n  \geq 1}$, along with the canonical surjections $\mathcal{S}^n(Q^{\psi})(F) \to \mathcal{S}^{n+1}(Q^{\psi})(F)$. The colimit of this diagram of sets is $\mathcal{L}(Q^{\psi})(F)$.  

Since $\pi_0^{\mathbb{A}^1}(Q^{\psi})(F) \to \mathcal{L}(Q^{\psi})(F)$ is a bijection for all finitely generated field extensions $F/k$ by \cite[Theorem 1.1]{BRS}, a major part of the proof of the Theorem \ref{main theorem for Q} reduces to showing that for any finitely generated field extension $F/k$, we have
    \[\mathcal L(Q^{\psi})(F):=\lim_{n}\mathcal S^n(Q^{\psi})(F)=\mathcal S^m(Q^{\psi})(F),\]
where $m=1$ when $\psi$ is isotropic and $m=2$ when $\psi$ is anisotropic.

For the sets $\mathcal{S}^n(Q^{\psi})(F)$ or $\pi_0^{\mathbb{A}^1}(Q^{\psi})(F)$ to be non-empty, it is necessary and sufficient that $Q^{\psi}$ should have an $F$-rational point. We observe that $Q^{\psi}$ has an $F$-rational point if and only if its projective completion, i.e. its closure in $\#P^n_k$, has an $F$-rational point. 

\begin{notation}\label{notaion for projective completion}
    Let $\bar{Q}^{\psi}$ denote the projective completion of $Q^{\psi}$ in $\mathbb{P}^n_k$, defined as the locus of zeroes of the homogeneous polynomial $\psi(x_1, \ldots, x_n)-x_{n+1}^2$. Let $H^{\psi}$ denote the hyperplane section at infinity of $\bar{Q}^{\psi}$, i.e. $H^{\psi} := \bar{Q}^{\psi}-Q^{\psi}$. The quadratic form $\psi$ corresponds to the projective quadric $H^{\psi}$.
\end{notation}

Since $\psi \perp \< -1 \>$ is the quadratic form corresponding to the projective quadric $\bar{Q}^{\psi}$, the isotropy of this form is detected by the  existence of a $k$-rational point on $\bar{Q}^{\psi}$, and hence also on $Q^{\psi}$. In other words, $Q^{\psi}$ has a $k$-rational point if and only if $\psi\perp\<-1\>$ is isotropic, which is equivalent to $1\in D(\psi)$. Thus, there exists a regular form $\phi$ in $n-1$ variables such that
\[ \psi \sim  \< 1\> \perp (-\phi),\quad \text{and}\quad \psi\perp\<-1\>\sim \<1,-1\>\perp (-\phi).\]
This form $\phi$ will play a crucial role in the proof of Theorem \ref{main theorem for Q}. (The minus sign in $(-\phi)$ is merely to ensure that $1$ is an element of $D(\phi)$ whenever $\psi$ is isotropic.)

As the statement of Theorem \ref{main theorem for Q} suggests, we will see that the number of iterations of $\mathcal{S}$ required for the stabilisation of field-valued sections of $\mathcal{L}(Q^{\psi})$ depends on the isotropy of $\psi$, or equivalently, on the existence of a $k$-rational point on $H^{\psi}$. 

The case when $\psi$ is isotropic is dealt with in subsection \ref{section isotropic case}. The results in this case will be useful for our investigation of the anisotropic case, which is presented in subsection \ref{section anisotropic case}. 

\subsection{The isotropic case}
\label{section isotropic case}

Let $\psi$ be an isotropic regular form in $n \geq 3$ variables. Then $\psi\perp\<-1\>$ is isotropic and we have $\psi \sim (-\phi) \perp \< 1 \>$ for some regular form $\phi$ in $n-1$ variables. 
We will prove that when $\psi$ is isotropic, the sequence $\{\mathcal{S}^n(Q^{\psi})(F)\}$ stabilizes at $n=1$ for any finitely generated field extension $F/k$. The proof will consist of three steps: 
\begin{itemize}
\item[(i)] We will first look at the case when $\phi$ is isotropic. We prove that, in this case, $Q^{\psi}$ is $\mathbb{A}^1$-chain connected, and hence $\mathbb{A}^1$-connected. 
\item[(ii)] In the general situation (i.e. with $\psi$ is isotropic but without any assumption on $\phi$), we will obtain an isomorphism $\mathcal{S}(Q^{\psi})(F) \xrightarrow{\sim} \frac{F^*}{\< D(\phi) \>}$ for any finitely generated field extension $F/k$.
\item[(iii)] We will then prove that $\mathcal{S}(Q^{\psi})(F) \to \mathcal{S}^2(Q^{\psi})(F)$ is a bijection for any finitely generated field extension $F/k$. 
\end{itemize}

Since any isotropic quadratic form is universal, it follows that if $\phi$ is isotropic, then $D(\phi)=k^*$. The following lemma establishes Theorem \ref{main theorem for Q}(1) in the case where $\phi$ is isotropic.  

\begin{lemma}\label{lemma on phi isotropic case}
    Let $\psi$ be an isotropic form in $n\geq 3$ variables over a field $k$ of characteristic 0. Then there exists a regular form $\phi'$ such that $\psi\sim\<1,-1\>\perp (\phi')$. Let $\phi=\<1\>\perp(-\phi')$ and $X^{\phi}$ denote the affine quadric in $\mathbb A^n_k$ which is the zero locus of polynomial $x_1x_2-\phi(1,x_3,\dots,x_n)$, where $x_1,\dots,x_n$ are the coordinates variables of $\mathbb A^n_k$. Then
    \begin{enumerate}
        \item $Q^{\psi}$ is isomorphic to $X^{\phi}$.
        \item Let $\pi:X^{\phi}\rightarrow \mathbb A^1_k$ be the projection on $x_1$-axis. Then the fiber $\pi^{-1}(a)$ is $\mathbb A^1$-chain connected for any $a\in k^*$ and $\pi^{-1}(0)$ has a $k$-rational point if and only if $\phi$ is isotropic.
        \item If $\phi$ is isotropic, then for all finitely generated field extensions $F/k$, we have  
    \[
    \mathcal{S}(Q^{\psi})(F)=*.
    \]  
    In other words, $Q^{\psi}$ is $\mathbb{A}^1$-chain connected.  
    \end{enumerate}
\end{lemma}

\begin{proof}
For (1), recall that $Q^{\psi}$ is defined as the locus of zeroes of polynomial $\psi(x_1,\dots,x_n)-1$ in $\mathbb A^n_k$. Since $\psi\sim\<1,-1\>\perp \phi'$, $Q^{\psi}$ is isomorphic to the affine quadric defined by $x_1x_2-\phi'(x_3,\dots,x_n)-1$, which is same as $x_1x_2-\phi(1,x_3,\dots,x_n)$.

For (2), the morphism $\pi: X^{\phi} \to \mathbb{A}^1_k$ is induced by the following  homomorphism of coordinate rings:
        \[k[x_1]\rightarrow k[x_1,\dots,x_n]/\< x_1x_2-\phi(1,x_3,\dots,x_n)\>. \]
For any $a \in k^*$, the fiber $\pi^{-1}(a)$ is the affine variety  
\[
\Spec k[x_2, \dots, x_n] / \< a x_2-\phi(1,x_3,\dots,x_n) \> \simeq \mathbb{A}^{n-2}_k,
\]  
which is $\mathbb{A}^1$-chain connected.  

The fiber $\pi^{-1}(0)$ is the affine variety  
\[
\Spec  \frac{k[x_2, \dots, x_n]}{\< \phi(1,x_3,\dots,x_n) \>} \cong \Spec k[x_2] \times \Spec  \frac{k[x_3, \dots, x_n]}{\< \phi(1,x_3,\dots,x_n) \>}.
\]  
Let $Y$ denote the quadratic hypersurface in $\mathbb{A}^{n-2}_k$ defined by $\phi(1,x_3,\dots ,x_n)=0$, hence $\pi^{-1}(0)=\mathbb A^1_k\times Y$. Clearly $\phi$ is  isotropic if and only if $Y(k) \neq \varnothing$. 

For (3), since $\mathcal{S}(Q^{\psi})(F)=\mathcal{S}(Q^{\psi}_F)(F)$ and $\psi_F$ is isotropic for all $F/k$, we may assume that $F=k$. Moreover since $Q^{\psi}\cong X^{\phi}$, it suffices to show that $\mathcal{S}(X^{\phi})(k)=*$ when $\phi$ is isotropic. For any $k$-rational point $y$ of $Y$, let $L_y$ denote the affine line in $\pi^{-1}(0)$ given as the image of the map  
\[
\mathbb{A}^1_k \xrightarrow{\text{id} \times y} \mathbb{A}^1_k \times Y.
\]  

    \noindent \textbf{Claim: }For any $y\in Y(k)$, there exists a section $s:\mathbb A^1_k\rightarrow X^{\phi}$ of $\pi$ such that $s(0)\in L_{y}$.
    \begin{proof}[Proof of the Claim]
Since $y \in Y(k)$, we have $y=(y_3, \dots, y_n)$ for some $y_i \in k$, and  
\[\phi(1,y_3,\dots,y_n)=1+\sum_{i=3}^n b_i y_i^2=0.\]
Define $s: \mathbb{A}^1_k \to X^{\phi}$ by  
\[t \mapsto (s_1(t),\dots, s_n(t))=(t, s_2(t), t + y_3, \dots, t + y_n),\]
where  
\[s_2(t)=\frac{\phi(1,s_3(t),\dots, s_n(t))}{s_1(t)}=\frac{1 + \sum_{i=3}^n b_i (t + y_i)^2}{t}.\]  
Expanding $s_2(t)$, we get  
\[s_2(t)=\frac{1 + \sum_{i=3}^n b_i y_i^2 + t f(t)}{t},\]  
where $f(t) \in k[t]$. Thus, $s_2(t)=f(t)$ is a polynomial in $t$, and hence $s$ is well defined.  

Therefore, $s$ is a section of $\pi$, and we have  
\[s(0)=(0, f(0), y_3, y_4, \dots, y_n)\in L_y. \qedhere 
\]\end{proof}
Since $y$ was an arbitrary $k$-rational point of $Y$, and any fiber of $\pi$ over a non-zero $k$-rational point is $\mathbb A^1$-chain connected, the above claim implies that $\mathcal S(X^{\phi})(k)=*$, as required.
\end{proof}

\begin{notation}\label{definitions for isotropic setting}
We fix the following setting for the remainder of this subsection.
\begin{enumerate}
    \item Let $\psi := \< 1, -1 \> \perp \phi'$ and $\phi := \< 1 \> \perp (-\phi')$, 
    where $\phi'$ is a fixed regular quadratic form over a field $k$ of characteristic $0$ in $n-2$ variables, with $n \ge 3$.
    \item We fix the coordinate system $(x_1, \dots, x_n)$ on $\mathbb{A}^n_k$. 
    With respect to this choice, we identify $Q^{\psi}$ with the affine hypersurface $X^{\phi}$ defined in Lemma~\ref{lemma on phi isotropic case}; that is, $Q^{\psi} \subset \mathbb{A}^n_k$
    is the zero locus of the polynomial $x_1x_2-\phi(1, x_3, \dots, x_n)$. Let $\pi:Q^\psi\rightarrow \mathbb A^1_k$ denote the projection on $x_1$-coordinate after identifying $Q^{\psi}$ with $X^{\phi}$.
    \item Let $C$ be an irreducible curve, and $f: C \dashrightarrow Q^{\psi}$ a rational map. By abuse of notation, we use the same letter $f$ to denote its restriction to the generic point of $C$. The induced homomorphism on coordinate rings is written $f^* : k[Q^{\psi}] \longrightarrow k(C)$, and we set $f_i := f^*(x_i)$ for $1 \le i \le n$.
\end{enumerate}
\end{notation}
\begin{theorem}\label{iteration 1 for isotropic case}
Let $\psi,\phi$ and $Q^{\psi}$ be as in Notation \ref{definitions for isotropic setting}. Then for any finitely generated field extension $F/k$, we have an isomorphism  
    \[
    \mathcal{S}(Q^{\psi})(F) \cong \frac{F^*}{\< D(\phi_F) \>}.
    \]  
\end{theorem}
\begin{proof}
As before, assume that $F=k$.  
If $\phi$ is isotropic, the result follows immediately from Lemma~\ref{lemma on phi isotropic case}.  
Hence, assume that $\phi$ is anisotropic.  

Let $\pi : Q^{\psi} \to \mathbb{A}^1_k$ denote the projection onto the $x_1$–coordinate as in Notation \ref{definitions for isotropic setting}, and let
$\pi_k : Q^{\psi}(k) \to \mathbb{A}^1_k(k)=k$ be the induced map on $k$–points.
Since $\phi$ is anisotropic, Lemma~\ref{lemma on phi isotropic case} implies that the image of $\pi_k$ is contained in $k^*$.
Write $\overline{\pi}_k$ for the composition
\[
 Q^{\psi}(k)\xrightarrow{\ \pi_k\ } k^*\longrightarrow \frac{k^*}{\<D(\phi)\>},
 \quad (x_1, \dots, x_n) \longmapsto x_1\<D(\phi)\>.
\]

\noindent\textbf{Claim 1:} $\overline{\pi}_k$ factors through the map $Q^{\psi}(k) \to \mathcal{S}(Q^{\psi})(k)$. 
\begin{proof}[Proof of Claim 1]
Let $f: \mathbb{A}^1_k \to Q^{\psi}$ be an $\mathbb A^1$-homotopy of $k$ in $Q^{\psi}$. Recalling from Notation~\ref{definitions for isotropic setting} that $f^*(x_i)=f_i$, we will now show that $f_1(t) \in c \< D(\phi_{k(t)}) \>$ for some $c \in k^*$. To do so, let $P$ be any closed point of $\mathbb{A}^1_k$ where $\phi_{k(P)}$ is anisotropic. We then have the following:
\[
\begin{aligned}
v_P(f_1 f_2) &= v_P\big(\phi(1, f_3, \dots, f_n)\big) \\
&= \min\{v_P(1), v_P(f_3^2(t)), \dots, v_P(f_n^2(t))\} \quad \text{(Lemma \ref{lemma for specialisation}(2))} \\
&= 0.
\end{aligned}
\]
Therefore, $v_P(f_1)=v_P(f_2)=0$. Specifically, $v_P(f_1)$ is even whenever $\phi_{k(P)}$ is anisotropic. In view of the QVT (Proposition \ref{QVT}), this implies that there exists  $c \in k^*$ such that 
\[
f_1(t) \in c \< D(\phi_{k(t)}) \>.
\]  
Since $\phi$ is anisotropic, we have $v_{t_1}(f_1)=0$ for any $t_1 \in k$. Hence, for any $t_1 \in K$, we have 
\[
f_1(t_1) \in c \< D(\phi) \>.
\]  
Thus, $\overline{\pi}_k$ factors through the map $Q^{\psi}(k) \to \mathcal{S}(Q^{\psi})(k)$ as required.
\end{proof}

We now have the following commutative diagram:
\[
\begin{tikzcd}
Q^{\psi}(k) \arrow[r, "{\overline{\pi}_k}"] \arrow[d] & \frac{k^*}{\< D(\phi) \>}  \\
 \mathcal{S}(Q^{\psi})(k) \arrow[ur, "\mathcal{F}_k", swap]
\end{tikzcd}
\]

The proof of the theorem will be completed once we show that $\mathcal{F}_k$ is bijective. Since the fiber of $\pi$ over any non-zero $k$-rational point is $\mathbb{A}^1$-chain connected (see Lemma \ref{lemma on phi isotropic case}), we conclude that $\mathcal{F}_k$ is surjective. Moreover, the injectivity of $\mathcal{F}_k$ reduces to the following claim:

\bigskip
\noindent\textbf{Claim 2: }For any $\lambda\in c\< D(\phi)\>$ and any $c\in k^*$ there exist $k$-rational points $P\in \pi^{-1}(c), Q\in \pi^{-1}(\lambda)$ such that $[P]_1=[Q]_1$.
\begin{proof}[Proof of Claim 2]
Given any $f:\mathbb A^1_k\rightarrow Q^{\psi}$, let $A(f)$ denote the following set
 \[A(f):=\{a\in k^*|\pi(f(t_1))=a \text{ for some } t_1\in k\}.\]

Define the following equivalence relation on $k^*$: $\lambda_1\sim \lambda_2$ if there exists a chain of a finite sequence of maps $f^i:\mathbb A^1_k\rightarrow X$ for $i=1,\dots,m$ such that $\lambda_1\in A(f^i)$ and $\lambda_2\in A(f^m)$ and $A(f^i)\cap A(f^{i+1})\neq\varnothing$ for $i<m$. The claim is further reduced to showing that $c\sim \lambda$ for any $c\in k^*$ and  any $\lambda\in c\< D(\phi)\>$. 

Let $c\in k^*$ and $\lambda\in c\< D(\phi)\>$ be fixed. We now construct the required sequence of rational curves.
Since $\lambda\in c\< D(\phi)\>$, there exists $\lambda_i\in D(\phi)$ such that 
\[\lambda=c\lambda_1\lambda_2\dots\lambda_j\]
 Any morphism $\mathbb A^1_k\xrightarrow f Q^{\psi}$ is given by $n$-polynomials $f_1(t), \dots, f_n(t) \in k[t]$
satisfying \[f_1(t)=\frac{\phi(1,f_3(t),\dots, f_n(t))}{f_2(t)},\]
and $\pi$ is the projection on $x_1$-coordinate.
\begin{case}[$j=1$]
     Hence, $\lambda=c\lambda_1$ where $\lambda_1=\phi(d,y_3,\dots,y_n)=d^2+\sum_{i=3}^{n}b_iy_i^2$ for some $d,y_i\in k$.
     First, we can assume that $d\neq 0$.  Indeed, suppose $d=0$, then
$y_i\neq 0$ for some $i$ say $y_3\neq 0$. Since $\phi$ is anisotropic, $b_3\neq -1$ and
   \[b_3y_3^2=\left (\frac{2b_3y_3}{1+b_3}\right )^2+b_3\left (\frac{(1-b_3)y_3}{1+b_3}\right )^2=d_1^2+b_3z_3^2.\]
Hence, 
\[\lambda_1=\phi(d_1,z_3,z_4\dots,z_n),\] where $d_1\neq 0$ and $z_i=y_i$ for $i\geq 4.$
We now show $\lambda\sim c$ in the following steps:\\

\item \textbf{STEP 1: $\lambda\sim cd_1^2$.} Let $f:\mathbb A^1_k\rightarrow Q^{\psi}$ be given by the following $n$-polynomials:
\[f_1(t):= cd_1^2\phi\left(1,\frac{tz_3}{d_1},\dots,\frac{tz_n}{d_1}\right),\quad f_2(t)=\frac{1}{cd_1^2}, \quad f_i(t)=\frac{tz_i}{d_1} \text{ for } i\geq 3.\] 
Since $f_1(0)=cd_1^2$ and $f_1(1)=c\lambda_1=\lambda$, we have $cd_1^2,\lambda\in A(f)$. Hence, $\lambda\sim cd_1^2$.
\item \textbf{STEP 2: $cd_1^2\sim c$.} Let $f:\mathbb A^1_k\rightarrow Q^{\psi}$ be given by the following $n$-polynomials:
\[f_1(t)=cd_1^2\phi(1,t,0,\dots,0), \quad f_2(t)=\frac{1}{cd_1^2},\quad f_3(t)=t,\quad f_n(t)=0 \text{ for }n>3.\]
 Let $g:\mathbb A^1_k\rightarrow Q^{\psi}$ be given by the following $n$-polynomials:
\[g_1(t)=c\phi(1,t,0,\dots,0), \quad g_2(t)=\frac{1}{c},\quad g_3(t)=t,\quad g_n(t)=0 \text{ for }n>3.\]
Then, $cd_1^2\in A(f_1)$ and $c\in A(f_2)$. Also, it is easy to find $a,b\in k$ such that 
\[f_1(a)=c(d_1^2+b_1d_1^2a^2)=c(1+b_1b^2)=g_1(b),\]
for instance by solving the 2 linear equations $d_1a-b=(1-d_1^2)/b_1$ and $d_1a+b=1$ for $a$ and $b$. Hence, $A(f)\cap A(g)\neq\varnothing$, which implies that $cd_1^2\sim c$.

Therefore, Step 1 and 2 together show that $c\sim c\lambda_1$.
\end{case}
\begin{case}[$j>1$]
     Suppose $\lambda=c\lambda_1\lambda_2\dots\lambda_j$.
Since $c\in k^*$ and $\lambda\in D(\phi)$ are arbitrary in Case $j=1$, we have
\[d\mu \sim d\text{ for any } d\in k^*,\, \mu\in D(\phi).\]
Since $\lambda_i\in D(\phi)$ for all $i$, applying the case $j=1$ repeatedly for $\mu:=\lambda_i$, we have
\[c\lambda_1\lambda_2\dots\lambda_j\sim c\lambda_1\lambda_2\dots\lambda_{j-1}\sim c\lambda_1\lambda_2\dots\lambda_{j-2}\sim\dots\sim c\lambda_1\sim c.\]
Therefore, $c\sim \lambda$ for any $\lambda\in c\< D(\phi)\>$ and we are done.\qedhere
\end{case}
\end{proof} 
\noindent Claim 2 implies that $\mathcal F_k$ is bijective, and hence the proof follows.
\end{proof}

The proof of Theorem~\ref{iteration 1 for isotropic case} provides an explicit isomorphism 
\(\mathcal{S}(Q^{\psi}_F)(F)\cong F^*/\<D(\phi_F)\>\) for any field extension $F/k$ such that $\phi_F$ is anisotropic.
For later reference, we record this construction as a definition.

\begin{definition}\label{def:FF}
Let $F/k$ be a finitely generated separable field extension such that $\phi_F$ is anisotropic.
We define
\[
\mathcal{F}_F : \mathcal{S}(Q^{\psi}_F)(F) \xrightarrow{\;\sim\;} \frac{F^*}{\< D(\phi_F) \>}
\]
to be the isomorphism constructed in the proof of Theorem~\ref{iteration 1 for isotropic case}, 
given explicitly by
\[
[(x_1, \dots, x_n)]_1 \longmapsto x_1 \, \< D(\phi_F) \>.
\]
\end{definition}

\begin{corollary}\label{good rational curves in X}
    Let $\phi$ be anisotropic and $f: \mathbb{A}^1_k \dashrightarrow Q^{\psi}$ be any rational curve in $Q^{\psi}$. Then the following are equivalent:
    \begin{enumerate}
        \item $f^*(x_1) \in c \< D(\phi_{k(t)}) \>$ for some $c \in k^*$.
        \item $[f]_1=[f(t_1)]_1$ in $\mathcal{S}(Q^{\psi})(k(t))$ for any $t_1 \in k$ where $f$ is defined. In other words, the map  
       $
        \Spec k(t) \xrightarrow{f} X
       $  
        is $\mathbb{A}^1$-homotopic to the constant map  
       $
        \Spec k(t) \xrightarrow{P} X,
       $ 
        where $P$ is any $k$-rational point in the image of $f$.
    \end{enumerate}
    Moreover, if $\phi$ is isotropic, (2) is satisfied by all rational curves in $X$.
\end{corollary}

\begin{proof}
    Suppose $\phi$ is anisotropic, and let $t_1 \in k$ be such that $f$ is defined on $t_1$. Let $f^{*}(x_1)=f_1\in k(t)$. Then
    \[
    (1)\iff f_1 \< D(\phi_{k(t)}\>= 
   (f_1(t_1))\< D(\phi_{k(t)}\> \iff \mathcal F_{k(t)}(f)=\mathcal F_{k(t)}(f(t_1))\iff (2).
    \]  
  If $\phi$ is isotropic, then $\phi_{k(t)}$ is isotropic and (2) holds because of Lemma \ref{lemma on phi isotropic case} in this case.
\end{proof}


\begin{lemma}\label{Lemma of curves for isotropic case}
    Let $(R, m, \kappa)$ be a DVR with function field $K$ and valuation $v$. Let $h \in Q^{\psi}(K)$ with $h^*(x_i)=h_i$, and suppose $\phi_{\kappa}$ is anisotropic.  
    \begin{enumerate}
        \item If $h$ lifts to $Q^{\psi}(\Spec R)$, then $v(h_1)=v(h_2)=0$.  
        \item If $h$ does not lift to $Q^{\psi}(\Spec R)$, then $v(h_1)$ or $v(h_2)$ is negative. Moreover,  
        \[
        \min \{v(h_1), v(h_2)\}=\min \{v(h_1), v(h_2), \dots, v(h_n)\}.
        \]  
    \end{enumerate}  
\end{lemma}

 \begin{proof}
Since $h\in Q^{\psi}(K)$ and $h_i=h^*(x_i)$, we have 
\[
h_1h_2=\phi(1, h_3, \dots, h_n).
\]  
Since $\phi_{\kappa}$ is anisotropic, Lemma \ref{lemma for specialisation}(1) gives  
\[
v(h_1h_2)=v(\phi(1, h_3, \dots, h_n))=\min\{v(1), v(h_3^2), \dots, v(h_n^2)\}.
\]  
Clearly, $h$ lifts to $Q^{\psi}(\Spec R)$ if and only if $v(h_i) \geq 0$ for all $i$, which, by the above relation, is equivalent to $v(h_1), v(h_2) \geq 0$.  

\begin{enumerate}
    \item Suppose $h$ lifts to $Q^{\psi}(\Spec R)$. Then $v(h_i) \geq 0$ for all $i$, so  
    \[
    v(h_1h_2)=\min\{v(1), v(h_3), \dots, v(h_n)\}=0.
    \]  
    Thus, $v(h_1)=v(h_2)=0$ as required.  

    \item Suppose $h$ does not lift to $Q^{\psi}(\Spec R)$. Then $v(h_1) < 0$ or $v(h_2) < 0$. We now show  
    \[
    \min \{v(h_1), v(h_2)\}=\min\{v(h_1), \dots, v(h_n)\}.
    \]  
    If $v(h_1h_2) \geq 0$, then $v(h_i) \geq 0$ for $i \geq 3$, and the equality follows.  
    If $v(h_1h_2) < 0$, then  
      \[
      \min\{v(1), v(h_3), \dots, v(h_n)\} < 0.
      \]  
      Hence,  
      \[
      v(h_1) + v(h_2)=v(h_1h_2)=2\min\{v(h_3), \dots, v(h_n)\}.
      \]  
      This implies  
      \[
      v(h_1) \leq \min\{v(h_3), \dots, v(h_n)\} \quad \text{or} \quad v(h_2) \leq \min\{v(h_3), \dots, v(h_n)\}.
      \]  
      Thus, the required equality holds.\qedhere
\end{enumerate}
\end{proof}

\begin{theorem}\label{iteration 2 for isotropic case}
Let $\psi$ be an isotropic form and $Q^{\psi}$ the corresponding quadric as in Notation~\ref{definitions for isotropic setting}.  
Then for any finitely generated field extension $F/k$, we have
\[
\mathcal{S}(Q^{\psi})(F)=\mathcal{S}^2(Q^{\psi})(F).
\]
\end{theorem}

\begin{proof}
    Assume that $F=k$. If $\phi$ is isotropic, then $\mathcal{S}(Q^{\psi})(k)=\mathcal{S}^2(Q^{\psi})(k)=*$ by Lemma \ref{lemma on phi isotropic case}(3). Hence, assume that $\phi$ is anisotropic. It suffices to show that whenever $Q_1, Q_2 \in Q^{\psi}(k)$ are connected by a 1-ghost homotopy, they must be $\mathbb{A}^1$-chain homotopic.

Let $\mathcal{H}:\mathbb{A}^1_k\rightarrow \mathcal{S}(Q^{\psi})$ be a 1-ghost homotopy given by the following data:
\[
\left( (U,V\xrightarrow{p} \mathbb{A}^1_k), W, h_U, h_V \right)
\]
such that $[pr_U^*(h_U)]_1=[pr_V^*(h_V)]_1$ in $\mathcal S(Q^{\psi})(W)$. We need to show that $ [h(P)]_1=[h(P')]_1 $ for any $ P, P' \in (U\coprod V)(k) $, where $ h=h_U \coprod h_V $. 

\bigskip
\noindent\textbf{Claim:} $ (h_U)_1 \in c\< D(\phi_{k(t)})\> $ for some $ c\in k^* $.

\begin{proof}[Proof of Claim]
Suppose the claim is false. Then, by QVT (Proposition \ref{QVT}), there exists a closed point $ P $ of $\mathbb{A}^1_k$ such that $ v_P((h_U)_1) $ is odd and $ \phi_{k(P)} $ is anisotropic. Therefore, $ v_P((h_U)_1) $ is nonzero. By Lemma \ref{Lemma of curves for isotropic case}, $ h_U $ is not defined at $ P $, implying $ P \notin U $.

Since $ U\coprod V $ is an elementary Nisnevich cover of $\mathbb{A}^1_k$, there exists $ P' \in V $ such that $ p^{-1}(P)=\{P'\} $ and $ k(P')=k(P) $. Let $ V' $ be an irreducible subscheme of $ V $ containing $ P' $, and let $ q := p|_{V'} $. Also, set $ W' := U \times_{\mathbb{A}^1_k} V' $. Since $P' \in V'$ and $\phi_{k(P')}$ is anisotropic, Lemma~\ref{lemma on isotropic forms over DVR} implies that $\phi_{k(W')}$, which is the same as $\phi_{k(V')}$, is also anisotropic. 

Since $[pr_U^*(h_U)]=[pr_V^*(h_V)]$ in $\mathcal S(Q^{\psi})(W)$, we have
\[
[(pr_U^*h_U)|_{W'}]=[h_{V'}|_{W'}] \quad \text{in } \mathcal{S}(Q^{\psi})(k(W')).
\]

Applying Theorem~\ref{iteration 1 for isotropic case} to the generic point of $W'$ (see also Definition~\ref{def:FF}), we obtain
\[
\mathcal{F}_{k(W')} \big( (pr_U^*h_U)|_{W'} \big)
= \mathcal{F}_{k(W')} \big( h_{V'}|_{W'} \big).
\]

 Since $(pr_U^*h_U)^*(x_1)=q^*(h_U^*(x_1))$, the above equality implies
 \[
q^*((h_U)_1)\< D(\phi_{k(W')}) \>=(h_{V'})_1\< D(\phi_{k(W')}) \>
\]
which is same as 
\[
q^*((h_U)_1)(h_{V'})_1
\in \< D(\phi_{k(W')}) \>.
\]

Since $ \phi_{k(P')} $ is anisotropic, Lemma \ref{lemma for specialisation} implies $ v_{P'}\big(q^*((h_U)_1)(h_{V'})_1\big)) $ is even. Because $ q $ is an unramified morphism, we have  
\[
v_{P'}(q^*((h_U)_1))=v_P((h_U)_1),
\]
which is odd.  Since $v_{P'}\big(q^*((h_U)_1)(h_{V'})_1\big)) $ is even and we showed that $v_{P'}(q^*((h_U)_1))$ is odd, hence $v_{P'}((h_{V'})_1) $ must be odd. By Lemma \ref{Lemma of curves for isotropic case}(1), $ h_{V'} $ cannot be defined at $ P' $, which is a contradiction. Therefore, the claim is proved.
\end{proof}

Thus,  $(h_U)_1 \in c \< D(\phi_{k(t)}) \> $ for some $ c \in k^* $, which further implies $ (h_U)_1(P) \in c\< D(\phi_k) \> $ for any $ P\in U(k) $. By Theorem \ref{iteration 1 for isotropic case}, we conclude that  
\[
[h(P)]_1=[h(P')]_1 \quad \text{for any } P,P'\in U(k).
\]

It remains to show that for any $ P\in V $, we have $ [h_V(P)]_1=[h_U(P')]_1 $ for some $ P'\in U(k) $. Let $ V' $ be any irreducible component of $ V $, and set $ W' := V \times_{\mathbb{A}^1_K} V_1 $. Define $ q=p|_{V'} $ and let $ h_{V'} := h_V|_{V'} $.

First, assume that $ \phi_{k(W')} $ is isotropic. Since $ k(W')=k(V') $, it follows that $ \phi_{k(p)} $ is isotropic for any closed point $ p \in V' $ by Lemma \ref{lemma on isotropic forms over DVR}. Given that $ \phi $ is anisotropic, this implies that $ V'(k)=\varnothing $, and we are done.

Applying Theorem \ref{iteration 1 for isotropic case} to the generic point of $ W' $, we obtain  
\[
q^*((h_U)_1)\< D(\phi_{k(W')}) \>=(h_{V'})_1\< D(\phi_{k(W')}) \>.
\]
Since $ (h_U)_1 \in c \< D(\phi_{k(t)}) \> $, we have $ (h_{V'})_1 \in c \< D(\phi_{k(W)}) \> $, which further implies  
\[
(h_{V'})_1(P) \in c\< D(\phi) \> \quad \text{for any } P\in V'(k).
\]
Since $ V' $ was an arbitrary irreducible component of $ V $, Theorem \ref{iteration 1 for isotropic case} then ensures that  
\[
[h(P)]_1=[h(P')]_1 \quad \text{for any } P,P'\in (U\coprod V)(k).
\]
This completes the proof.
\end{proof}

\begin{proof}[Proof of Theorem \ref{main theorem for Q} (1)]
By Theorem \ref{iteration 2 for isotropic case} and Lemma \ref{lemma on stabilisation}, we have
\[\mathcal L(Q^{\psi})(F)=\mathcal S(Q^{\psi})(F) \quad \forall F/k.\] 
Hence, by Theorem \ref{iteration 1 for isotropic case} and \cite[Theorem 1.1]{BRS}, for any $F/k$, we have
\[\pi_0^{\mathbb A^1}(Q^{\psi})(F)=\mathcal S(Q^{\psi})(F)=\frac{F^*}{\< D(\phi_F)\>}.\qedhere\]
\end{proof}

\subsection{The anisotropic case}
\label{section anisotropic case}
Let $\psi$ be an anisotropic quadratic form over a field $k$ of characteristic $0$ in $n \ge 3$ variables. Assume that $Q^{\psi}(k) \neq \emptyset$, and hence $\bar{Q}^{\psi}(k) \neq \emptyset$. 
This implies that $\psi \perp \< -1 \>$ is isotropic, which in turn forces $1 \in D(\psi)$. 
Therefore, as in the isotropic case, there exists a quadratic form $\phi$ in $n-1$ variables such that 
\[
\psi \sim \< 1 \> \perp (-\phi).
\]

\begin{notation}\label{definition for anisotropic setting}
   Let $\psi := \< 1 \> \perp (-\phi)$ be an anisotropic form over a field $k$ of characteristic $0$ in $n \ge 3$ variables. 
   We fix the embedding $Q^{\psi} \hookrightarrow \mathbb{A}^n_k$ given as the zero locus of the polynomial $\psi(x_1,\dots,x_n)-1$ which is same as $x_1^2-\phi(x_2, \dots, x_n)-1$, where $x_1, \dots, x_n$ are the coordinate functions on $\mathbb{A}^n_k$.
\end{notation}

\begin{remark}\label{SY remark}
    Let $F$ be a finitely generated field extension of $k$. The following are two straightforward consequences if we assume $\psi_F$ to be anisotropic.
    \begin{enumerate}
        \item Any morphism $f: \mathbb{A}^1_F \to Q^{\psi}$ must be constant. Therefore, $\mathcal{S}(Q^{\psi})(F)=Q^{\psi}(F)$.  
        \item The complement $\bar{Q}^{\psi} \setminus Q^{\psi}$ has no $F$-rational points. Consequently, any rational map $f: \mathbb{A}^1_F \dashrightarrow Q^{\psi}$ extends to a map defined at all $F$-points of $\mathbb{A}^1_F$. 
    \end{enumerate}
\end{remark}

We now prove Theorem~\ref{main theorem for Q}(2), namely that for any finitely generated separable field extension $F/k$,
\[
\mathcal{S}^2(Q^{\psi})(F)=\mathcal{S}^3(Q^{\psi})(F)=\mathcal{L}(Q^{\psi})(F).
\] The case of fields $F$ when $\psi_F$ is isotropic has already been treated in Subsection~\ref{section isotropic case}. It remains to handle the case where $\psi_F$ is anisotropic. From Remark~\ref{SY remark}, $\mathcal S(Q^{\psi})(F)$ is known for all such $F/k$.
Theorem~\ref{main theorem for Q}(2) will be proved in two steps:
\begin{itemize}
    \item[(i)] We will give a geometric description of the set $\mathcal{S}^2(Q^{\psi})(F)$ for $F$ when $\psi_F$ is anisotropic.
    \item[(ii)] Then, we will show that $\mathcal{S}^2(Q^{\psi})(F)=\mathcal{S}^3(Q^{\psi})(F)$ for any field extension $F/k$, ensuring that stabilization of $\{ \mathcal{S}^n(Q^{\psi})(F)\}$ occurs at $n=2$.
\end{itemize}

The following lemma is analogous to Lemma \ref{Lemma of curves for isotropic case}.

\begin{lemma}\label{curve lemma for anisotropic case}
    Let $(R, m, \kappa)$ be a DVR with function field $K$ and valuation $v$. Let $f \in Q^{\psi}(K)$ with $f^*(x_i)=f_i$, and suppose $\phi_{\kappa}$ is anisotropic.  
    \begin{enumerate}
        \item If $f$ lifts to $Q^{\psi}(\Spec R)$, then both $v(f_1-1)$ and $v(f_2-1)$ are even.  
        \item If $f$ does not lift to $Q^{\psi}(\Spec R)$, then $v(f_1)<0$ and
        \[
       v(f_1)=\min \{v(f_2), \dots, v(f_n)\}.
        \]  
    \end{enumerate}  
\end{lemma}

\begin{proof}
We have $f_1^2-1=\phi(f_2, f_3, \dots, f_n)$. Since $\phi_{\kappa}$ is anisotropic, we can deduce
\[
v(f_1^2-1)=v(\phi(f_2, f_3, \dots, f_n))=\min\{v(f_2^2), \dots, v(f_n^2)\}.
\]
We will use the above relation to prove the lemma.
\begin{enumerate}
    \item From the above relation, we have that $v(f_1^2-1)$ is even. Suppose $f$ lifts to $Q^{\psi}(\Spec R)$, then $v(f_1-1) \geq 0$ and $v(f_1 + 1) \geq 0$. Both cannot be positive simultaneously; hence $v(f_1-1)=v(f_1^2-1)$ or $0$, both of which are even. The same is true for $v(f_1 + 1)$, and hence, we are done.
    \item Suppose $v(f_1) \geq 0$. Then we have $v(f_1^2-1) \geq 0$. The above relation implies that $v(f_i) \geq 0$ for all $i$, hence $f$ must $f$ lifts to $Q^{\psi}(\Spec R)$. Therefore, we have that if $f$ does not lift to $Q^{\psi}(\Spec R)$, then $v(f_1) < 0$. If $v(f_1) < 0$, then from the above relation, we get
    \[
    v(f_1^2)=v(f_1^2-1)=\min\{v(f_2^2), \dots, v(f_n^2)\},
    \]
    hence we arrive at our result.\qedhere
\end{enumerate}
\end{proof}

We now introduce the notion of \emph{good rational curves} in $Q^{\psi}$, which will play a key role in the computation of the sets $\mathcal{S}^2(Q^{\psi})(F)$. For any rational curve $f: \mathbb{A}^1_k \dashrightarrow Q^{\psi}$ and $f^*(x_i)=f_i$, we have $(f_1-1)(f_1 + 1)\big)=\phi(f_2, \dots, f_n)$.

\begin{definition}
A rational map $f: \mathbb{A}^1_k \dashrightarrow Q^{\psi}$ is called a \textit{good rational curve} in $Q^{\psi}$ if $f$ is non-constant and there exists $c \in k^*$ such that $f^*(x_1)-1 \in c \< D(\phi_{k(t)}) \>,$ or equivalently, $f^*(x_1) + 1 \in c \< D(\phi_{k(t)}) \>.$
\end{definition}

\begin{construction}\label{construction of isomorphism between Z and YF}
Let $\psi$ and $Q^{\psi}$ be as in Notation~\ref{definition for anisotropic setting}. 
Let $F/k$ be a field such that $\psi_F$ is isotropic. 
Then $\psi_F=\< 1 \> \perp (-\phi_F)$ with $1 \in D(\phi_F)$. 
Hence, there exist regular quadratic forms $\theta^F$ and $\theta'^F$ over $F$ such that 
\[
\phi_F \sim \theta^F, \qquad \text{and} \qquad \theta^F=\< 1 \> \perp \theta'^F.
\]
Similar to Notation \ref{definitions for isotropic setting}, let $X^{\theta^F} \subset \mathbb{A}^n_F$ denote the affine quadric defined by the equation $
z_1 z_2-\theta^F(1, z_3, \dots, z_n)=0$, where $z_i$ are coordinate variables of $\mathbb A^n_F$. Then $Q^{\psi}_F$ is isomorphic to $X^{\theta^F}$.

We now construct an explicit isomorphism 
\[
\mathcal{J}^F: Q^{\psi}_F \xrightarrow{\;\sim\;} X^{\theta^F}.
\]
On coordinate rings, the map $\mathcal{J}^F$ corresponds to the composition of the following two ring isomorphisms:

\[
\frac{F[z_1, \dots, z_n]}{\< z_1 z_2-\theta^F(1, z_3, \dots, z_n) \>}
\;\xrightarrow{\;\sim\;}\;
\frac{F[w_1, \dots, w_n]}{\< w_1^2-w_2^2-\theta^F(1, w_3, \dots, w_n) \>},
\]
given by
\[
z_1 \mapsto w_1-w_2, \qquad 
z_2 \mapsto w_1 + w_2, \qquad 
z_i \mapsto w_i \ \text{for } i \ge 3;
\]
and
\[
\frac{F[w_1, \dots, w_n]}{\< w_1^2-1-\theta^F(w_2, \dots, w_n) \>}
\;\xrightarrow{\;\sim\;}\;
\frac{F[x_1, \dots, x_n]}{\< x_1^2-1-\phi_F(x_2, \dots, x_n) \>},
\]
defined by
\[
w_1 \mapsto x_1, \qquad 
w_i \mapsto \sum_{j=2}^n a_{ij} x_j \quad \text{for } i \ge 2,
\]
for suitable $a_{ij} \in F$.
\noindent
(Note that $w_1^2-w_2^2-\theta^F(1, w_3, \dots, w_n)
= w_1^2-1-\theta^F(w_2, \dots, w_n)$, 
and the coefficients $a_{ij}$ exist because $\theta^F$ and $\phi_F$ are isometric over $F$.)
\end{construction}

\begin{remark}
For fields $F$ with $\psi_F$ isotropic, the quadric $X^{\theta^F}$ and fits into the framework of 
Notation~\ref{definitions for isotropic setting} in the isotropic case. Here, the roles of $k$, $\psi$, and $\phi$ in 
Notation~\ref{definition of quadrics} are played by $F$, $\< 1 \> \perp (-\theta^F)$, and $\theta^F$, respectively. Hence, all the results of Subsection~\ref{section isotropic case} apply to $X^{\theta^F}$, and therefore also to $Q^{\psi}_F$ via the isomorphism constructed above.
\end{remark}

    \begin{lemma}\label{lemma relating isotropic and anisotropic case}
    Let $\psi$ be an anisotropic form in $n\geq 3$ variables over a field of characteristic 0 as in Notation \ref{definition for anisotropic setting}. Let $f: \mathbb{A}_k^1 \dashrightarrow Q^{\psi}$ be a rational curve in $Q^{\psi}$, and let $F/k$ be a field extension such that $\psi_F$ is isotropic. Let $\theta^F$, $X^{\theta^F}$, and $\mathcal J^F$ be as defined in Construction \ref{construction of isomorphism between Z and YF}. Let $g: \mathbb{A}_F^1 \dashrightarrow X^{\theta^F}$ be a rational map defined as the composition
    \[
    \mathbb{A}_F^1 \overset{f_F}{\dashrightarrow} Q^{\psi}_F \xrightarrow{\mathcal J^F} X^{\theta^F}.
    \]
    \begin{enumerate}
        \item If $P \in \mathbb{A}_F^1$ such that $\phi_{F(P)}$ is anisotropic and $f_F$ is not defined at $P$, then $v_P(f_F^*(x_1-1))$ is odd if and only if $v_P(g^*(z_1))$ is odd.
        \item If $f$ is a good rational curve and $P$ is an $F$-rational point in the image of $f_F$, then the map $f_F: \Spec F(t) \to Q^{\psi}_F$ is chain $\mathbb{A}^1$-homotopic to the constant map $\Spec F(t) \xrightarrow{P} Q^{\psi}_F$. 
    \end{enumerate}
\end{lemma}

\begin{proof}

Suppose $f^*(y_i)=f_i\in k(t)$ and $g^*(z_i)=g_i\in F(t)$. Using the description of $\mathcal J^F$ from Construction \ref{construction of isomorphism between Z and YF}, we have
\[g_1=f_1-\sum_{j=2}^n a_{2,j}f_j,\, g_2=f_1+\sum_{j=2}^n a_{2,j}f_j \text{ and } g_i=\sum_{j=2}^n a_{i,j}f_j \text{ for }i>2.\]
Hence, $v_P(g_1)\leq \min\left\{v_P(f_1),v_P\left(\sum_{j=2}^n a_{2,j}f_j\right)\right\}.$

For (1), let $P \in \mathbb{A}_F^1$ such that $\phi_{F(P)}$ is anisotropic and $f_F$ is not defined at $P$. Then by Lemma \ref{curve lemma for anisotropic case}, we have $v_P(f_1)<0$ and 
\[v_P(f_1-1)=v_P(f_1)=\min\{v_P(f_2),\dots, v_P(f_n)\}.\]
Therefore,
\[v_{P}(f_1)\leq v_{P}(\sum_{j=2}^n a_{2,j}f_j).\]
It suffices to show that $v_P(f_1)$ is odd if and only if $v_P(g_1)$ is odd.
If the above inequality is strict, we will have $v_{P}(g_1)=v_{P}(f_1)$ and hence we are done. Otherwise, there exists $u,v\in \mathcal O_{P,\mathbb A^1_F}^*$ such that $f_1=\pi^ru$ and $\sum_{j=2}^n a_{2,j}f_j=\pi^rv$,
where $\< \pi\>=m_{P}$  and $r$ is a negative integer. Since $v_P(g_1)=v_P(\pi^r(u-v))$ and $v_P(f_1)=r$, what remains is to show that $v_P(u-v)$ is even.

Now $g_1g_2\in D(\phi_{F(t)})$ and $\phi_{F(P)}$ is anisotropic, hence $v_{P}(g_1g_2)$ is even.  Therefore, 
we have 
\[
\begin{aligned}
v_{P}(g_1g_2)&=v_{P}(\pi^r(u-v))+v_{P}(\pi^r)(u+v)\\
&=2r +v_{P}(u+v)+v_{P}(u-v)
\end{aligned}
\]
Since $v_{P}(g_1g_2)$ is even and $v_{P}(u+v)$ and $v_{P}(u-v)$ cannot be positive simultaneously, we have $v_{P}(u-v)$, $v_{P}(u+v)$ are both even which finishes the proof. 

For (2), suppose $f$ is a good rational curve in $Q^{\psi}$. Then there exists $c \in k^*$ such that $f^*(x_1-1) \in c \< D(\phi_{k(t)}) \>$. Hence,
\[
f_F^*(x_1-1) \in c \< D(\phi_{F(t)}) \>.
\]
Therefore, by QVT, we have that $v_P(f_F^*(x_1-1))$ is even for all $P \in \mathbb{A}_F^1$ whenever $\phi_{F(P)}$ is anisotropic. Let $P \in \mathbb{A}_F^1$ such that $\phi_{F(P)}$ is anisotropic.  If $f_F$ is not defined at $P$, then $v_P(g^*(z_1))$ is even by (1). If $f_F$ is defined at $P$, then so is $g$ and by Lemma \ref{Lemma of curves for isotropic case} $v_P(g^*(z_1))=0$.  Therefore, applying QVT to $g^*(z_1)$, there exists $d\in F^*$ such that
\[
g^*(z_1) \in d \< D(\phi_{F(t)}) \>.
\]
Now, (2) follows directly from Corollary \ref{good rational curves in X}.
\end{proof}

We now turn to a geometric description of $\mathcal{S}^2(Q^{\psi})(k)$ in the anisotropic case. 
The next two theorems show that this set can be described completely in terms of good rational curves on $Q^{\psi}$.

\begin{theorem}\label{Almost A1}
    Let $\psi$ be an anisotropic form as in Notation \ref{definition for anisotropic setting}. Let $f:\mathbb A^1_k\dashrightarrow  Q^{\psi}$ be a good rational curve in $Q^{\psi}$. Then $f$ induces a 1-ghost homotopy $\mathbb A^1_k\rightarrow \mathcal S(Q^{\psi})$ such that $[f(P)]_2=[f(P')]_2$ in $\mathcal S^2(Q^{\psi})(k)$ for any $P,P'\in k$.
\end{theorem}
\begin{proof}
We begin by constructing the Nisnevich cover required for the 1-ghost homotopy. Let $U$ be the maximal open subset of $\mathbb{A}^1_k$ on which $f$ is defined. Define the closed subset  
\[
T := \mathbb{A}^1_k \setminus U=\{Q_1, \dots, Q_r\}.
\]  
For each $ Q_i \in T $, we construct \'{e}tale morphisms $ p_i: V_i \to \mathbb{A}^1_k $ such that there exists $ R_i \in V_i $ satisfying $ p_i^{-1}(Q_i)=\{R_i\} $ and $ k(R_i)=k(Q_i) $.  

Since $ \operatorname{char}(k)=0 $, the field extension $ \operatorname{Spec} k(Q_i) \to \operatorname{Spec} k $ is \'{e}tale. Let $ \pi_i $ denote the pullback morphism  
\[
\pi_i: \mathbb{A}^1_{k(Q_i)} \to \mathbb{A}^1_k.
\]  
There exists a closed point $ R_i \in \pi_i^{-1}(Q_i)$ such that $k(Q_i)=k(R_i) $. Choose $V_i$ to be an open subset of $\mathbb{A}^1_{k(Q_i)}$ containing $R_i$ such that $\pi_i^{-1}(Q_i) \cap V_i=\{R_i\} $.  

Define  
\[
V := \coprod_{i=1}^r V_i, \quad p_i := \pi_i|_{V_i}, \quad p := \coprod_{i=1}^r p_i.
\]  
Then,  
\[
(U \hookrightarrow \mathbb{A}^1_k, \quad V \xrightarrow{p} \mathbb{A}^1_k)
\]  
forms an elementary Nisnevich cover of $\mathbb{A}^1_k $.  For each $i=1, \dots, r $, consider the following pullback square:  
\[
\begin{tikzcd}
	W_i \arrow[r, "pr_{2,i}"] \arrow[d, "pr_{1,i}"'] & V_i \arrow[d, "p_i"] \\
	U \arrow[r, hook] & \mathbb{A}^1_k
\end{tikzcd}
\]  

Denoting $k(Q_i)$ by $K_i$, we observe that $W_i$ is an open subscheme of $V_i$, and hence $k(W_i)=K_i(t) $.  

To construct the desired morphism $\mathbb{A}^1_k \to \mathcal{S}(Q^{\psi}) $, it suffices, by Lemma \ref{lemma on lines of brs}, to construct morphisms  
\[
h_i: V_i \to Q^{\psi}, \quad i=1, \dots, r
\]  
such that  
\[
[pr_{1,i}^*(f)]_1=[pr_{2,i}^*(h_i)]_1  \quad \text{in } \mathcal{S}(Q^{\psi})(K_i(t)).
\]

Since $f$ cannot be defined at $Q_i$, we have $\psi_{K_i}$ is isotropic (see Remark \ref{SY remark}). Pick $P_i \in Q^{\psi}_{K_i}(K_i)$ such that $P_i$ is in the image of $f_{K_i}$. Since $f$ is a good rational curve in $Q^{\psi}$, by Lemma \ref{lemma relating isotropic and anisotropic case}, the maps
\[
\Spec K_i(t)\rightarrow U_{K_i} \xrightarrow{f_{K_i}} Q^{\psi}_{K_i}\quad\text{and}\quad \Spec K_i(t)\rightarrow \mathbb{A}^1_{K_i}\xrightarrow{P_i} Q^{\psi}_{K_i}
\]
have the same image in $\mathcal{S}(Q^{\psi}_{K_i})(K_i(t))$.

Hence, the maps obtained after composing the above maps with the projection map to $Q^{\psi}$,
\[
\Spec K_i(t)\xrightarrow{f_{K_i}} Q^{\psi}_{K_i}\rightarrow Q^{\psi} \quad\text{and}\quad \Spec K_i(t)\xrightarrow{P_i} Q^{\psi}_{K_i}\rightarrow Q^{\psi}
\]
have the same image in $\mathcal{S}(Q^{\psi})(K_i(t))$.
Define $h_i$ to be the constant map which maps $V_i$ to the point $P_i$, that is $h_i$ is defined as the following composition:
\[
V_i\hookrightarrow \mathbb{A}^1_{K_i}\xrightarrow{P_i} Q^{\psi}_{K_i}\rightarrow Q^{\psi}.
\]
Since we have the following commutative diagram,
\[
\begin{tikzcd}
	W_i \arrow[rd,swap,rightarrow, "pr_{1,i}"] \arrow[r,hookrightarrow, ""] &U_{K_i} \arrow[r,rightarrow, "f_{K_i}"] \arrow[d,swap,rightarrow, ""] & Q^{\psi}_{K_i}\arrow[d,rightarrow, ""] \\
			& U \arrow[r,swap,rightarrow, "f"] &  Q^{\psi}
\end{tikzcd}
\]
the following morphisms,
\[
\Spec K_i(t)\xrightarrow{pr^*_{1,i}(f)}  Q^{\psi}\quad \text{and}\quad \Spec K_i(t)\xrightarrow{f_{K_i}} Q^{\psi}_{K_i}\rightarrow Q^{\psi}
\]
are the same in $\mathcal{S}(Q^{\psi})(K_i(t))$.
Thus, $[pr_{1,i}^*(f)]_1=[pr_{2,i}^*(h_i)]_1$ in  $\mathcal{S}(Q^{\psi})(K_i(t))$.
Hence, by Lemma \ref{lemma on lines of brs}, $(f, \coprod_{i=1}^{r}h_i)$ can be glued to give the required morphism $\mathcal{H}:\mathbb{A}^1_k\rightarrow \mathcal{S}(Q^{\psi})$, and the result follows.
\end{proof}

\begin{remark}  
    The proof of Theorem \ref{Almost A1} is the only place where we need $\cha k=0$. Otherwise, it suffices to assume $\cha k \neq 2$.  
\end{remark}  

\begin{definition}
    Let $\sim_{\phi_F}$ denote the following equivalence relation on $Q^{\psi}(F)$. Let $P,Q\in Q^{\psi}(F)$. Then, $P\sim_{\phi_F} Q$ if and only if there exist good rational curves $f_1,\dots,f_r$ in $Q^{\psi}$ and $t_0,\dots,t_r\in F$ such that $f_1(t_0)=P$, $f_i(t_i)=f_{i+1}(t_{i})$, and $f_{r}(t_r)=Q$.
\end{definition}
\begin{theorem}\label{iteration 2 for anisotropic}
Suppose $\psi_F$ is anisotropic for some $F/k$. Then $\mathcal S^2(Q^{\psi})(F)=Q^{\psi}(F)/\sim_{\phi_F}$.
\end{theorem}

\begin{proof}
Assume $F=k$. Suppose $Q_1 \sim_{\phi_k} Q_2$. Then, there exists a chain of good rational curves $f_i: \mathbb{A}^1_k \dashrightarrow Q^{\psi}$ for $i=1, \dots, m$ and $t_i \in k$ for $i=0, 1, \dots, m$ such that $f_1(t_0)=Q_1$, $f_i(t_i)=f_{i+1}(t_{i})$, and $f_m(t_m)=Q_2$. Hence, by Theorem \ref{Almost A1}, $[Q_1]_2=[Q_2]_2$ in $\mathcal{S}^2(Q^{\psi})(k)$.

Conversely, suppose $[Q_1]_2=[Q_2]_2$ in $\mathcal{S}^2(Q^{\psi})(k)$ are 1-ghost homotopic. We will show that $Q_1$ and $Q_2$ lie on a good rational curve. Let $\mathcal{H}:\mathbb{A}^1_k\rightarrow \mathcal{S}(Q^{\psi})$ be a 1-ghost homotopy given by the following data:
\[
\left( (U,V\xrightarrow{p} \mathbb{A}^1_k), W, h_U, h_V, P_1, P_2 \right)
\]
 with $h= h_U\coprod h_V$ such that $h(P_i)=Q_i$ and $[pr_U^*(h_U)]_1=[pr_V^*(h_V)]_1$ in $\mathcal S(Q^{\psi})(W)$.

Let $V'$ be any irreducible component of $V$ with $W'=U\times_{\mathbb A^1_k}V'$. Since there is a non-constant $\mathbb{A}^1_{k(W')}$ in $Q^{\psi}$ in the above data, we have $\psi_{k(W')}$ is isotropic (see Remark \ref{SY remark}). Hence, $\psi_{k(P)}$ is isotropic for any $P \in V'$, which implies $V'(k)=\varnothing$.  Therefore, $P_1, P_2 \in U(k)$ and the proof will be complete if we show that $h_U$ is a good rational curve.

We need to show that $h_U^*(x_1-1)\in c\<D(\phi_{k(t)})\>$ for some $c\in k^*$. Suppose not. Then by QVT, there exists a closed point $R$ of $\mathbb{A}^1_k$ such that $v_R(h_U^*(x_1-1))$ is odd and $\phi_{k(R)}$ is anisotropic. Lemma \ref{curve lemma for anisotropic case} implies that $h_U$ is not defined at $R$ and $v_R(h_U^*(x_1-1))$ is a negative integer.  

Since $U \coprod V$ is an elementary Nisnevich cover, there must exist a point $R' \in V$ such that $k(R')=k(R)=K(say)$ and $p(R')=R$. Let $V'$ be the irreducible component containing $R'$. We will show that $h_{V'}$ can not be defined on $R'$, hence we will arrive at a contradiction and the proof will be complete. 
Let $F^h$ be the fraction field of the ring $\mathcal O_{R',V'}^h$. Consider the following commutative diagram
\[
\begin{tikzcd}
	\Spec F^h \arrow[r] \arrow[d]& W' \arrow[r, "pr_{1}"] \arrow[d, "pr_{2}"] & U \arrow[d] \\
	\Spec \mathcal O_{R',V'}^h \arrow[r]& V' \arrow[r, "p"] & \mathbb{A}^1_k
\end{tikzcd}
\]  
Consider the  following restriction of $h_U$ and $h_V$ to $\text{Spec } F^h$ which are given by the following composition:
\[\text{Spec } F^h\rightarrow U\xrightarrow{h_U}Q^{\psi}, \quad \text{Spec } F^h\rightarrow V' \xrightarrow{h_V} Q^{\psi}.\]
 
Now $K\subset \mathcal O_{R',V'}^h$, hence the above maps induce the following maps 
\[H_U,H_{V'}:\text{Spec }F^h\rightarrow Q^{\psi}_K.\]

Following are some direct consequences of the above construction:
\begin{enumerate}[label=(\alph*)]
\item Since $v_R(h_U^*(x_1-1))$ is odd and the maps $p$ and $ \text{Spec }\mathcal O_{R',V'}^h\rightarrow \text{Spec }\mathcal O_{R',V'}$ are unramified, it follows that $v_{R'}(H_U^*(x_1-1))$ is odd.
\item $[pr_U^*(h_U)]_1=[pr_V^*(h_{V'})]_1$ in $\mathcal S(Q^{\psi})(W')$ implies that
\[[H_U]_1=[H_{V'}]_1  \text{ in }  \mathcal S(Q^{\psi}_K)(F^h).
\]
\end{enumerate}

Since $H_{V'}:\Spec F^h\rightarrow Q^{\psi}_K$ is defined through the two morphisms
\[\text{Spec } F^h\rightarrow V' \xrightarrow{h_V} Q^{\psi} \text{ and }\text{Spec } F^h\rightarrow \Spec K,\]
and $R'\in V'$, the map $H_{V'}$ must factor through the map $\Spec F^h\rightarrow \Spec \mathcal O_{R',V}^h$. We will use (a) and (b) to show that $H_{V'}$ does not factor through the map $\Spec F^h\rightarrow \Spec \mathcal O_{R',V}^h$, leading to a contradiction. 

By Construction \ref{construction of isomorphism between Z and YF}, since $\psi_K$ is isotropic, there exist quadratic form $\theta^K$ over $K$ such that 
\[Q^{\psi}_K\xrightarrow{\mathcal J^K} X^{\theta^K}:= \text{Spec } \frac{K[z_1,\dots, z_n ]}{\< z_1z_2-\theta^K(1,z_3,\dots,z_n)\>}\]
is an isomorphism.
Define $G_U$ and $G_{V'}$ by the following composition:
\[\Spec F^h\xrightarrow{H_{\bullet}}Q^{\psi}_K\xrightarrow{\mathcal J^K} X^{\theta^K}.\]

\noindent \textbf{Claim: }$G_{V'}$ does not factor through the map $\Spec F^h\rightarrow \Spec \mathcal O_{R',V'}^h$. 

\begin{proof}[Proof of Claim]
(b) implies that 
\[[G_U]_1=[G_{V'}]_1 \text{ in } \mathcal S(X^{\theta^K})(F^h).\]
Since $v_{R'}$ is a discrete valuation of $F^h$ with residue field $k(R')$ and $\phi_{k(R')}$ is anisotropic, it follows that $\phi_{F^h}$ is anisotropic by Lemma \ref{lemma on isotropic forms over DVR}. Hence, applying Theorem \ref{iteration 1 for isotropic case} on $F^h$-rational points for the quadric $X^{\theta^K}$, we obtain
\[G_U^*(z_1)G_{V'}^*(z_1)\in \< D(\phi_{F^h})\>.\]
Since $v_R(H_U^*(x_1-1))$ is odd by (a), similar computations as in Lemma \ref{lemma relating isotropic and anisotropic case}(1) show that $v_{R'}(G_U^*(z_1))$ must be odd. Since $\phi_{k(R')}$ is anisotropic, $v_{R'}(G_U^*(z_1)G_{V'}^*(z_1))$ is even and $v_{R'}(G_U^*(z_1))$ is odd, we have $v_{R'}(G_{V'}^*(z_1))$ must be odd. 

This further implies that $v_{R'}(G_{V'}^*(z_1))$ is nonzero; hence, by Lemma \ref{Lemma of curves for isotropic case}(1), we obtain $v_{R'}(G_{V'}^*(z_i))<0$ for some $i$. Therefore, the map $G_{V'}: \Spec F^h\rightarrow X^{\theta^K}$ does not factor through the map $\Spec F^h\rightarrow \Spec \mathcal O_{R',V'}^h$.
\end{proof}

The above claim implies that $H_{V'}:\Spec F^h\rightarrow Q^{\psi}_K$ cannot factor through the map $\Spec F^h\rightarrow \Spec \mathcal O_{R',V'}^h$, leading to a contradiction. Therefore, $h_U$ must be a good rational curve, and the proof is complete.
\end{proof}

We now proceed with the last step of Theorem \ref{main theorem for Q}(2).

\begin{theorem}\label{iteration 3 for anisotropic case}
Let $\psi$ and $Q^{\psi}$ be as in Notation~\ref{definition for anisotropic setting}. 
Then, for every finitely generated field extension $F/k$, we have
\[
\mathcal{S}^2(Q^{\psi})(F)=\mathcal{S}^3(Q^{\psi})(F).
\]
\end{theorem}

\begin{proof}
If $\psi_F$ is isotropic, then by Theorem \ref{main theorem for Q}(1), we have $\mathcal{L}(Q^{\psi})(F)=\mathcal{S}(Q^{\psi})(F)$, and the assertion follows. Now, assume that $\psi_F$ is anisotropic and, by a base change, further assume that $F=k$.  

Suppose that $Q_1, Q_2 \in Q^{\psi}(k)$ are connected by a 2-ghost homotopy, which consists of the following data:  
\[
\left( (U, V \xrightarrow{p} \mathbb{A}^1_k), W, h_U, h_V, P_1, P_2 \right)
\]
where $h=h_U \coprod h_V$ satisfies $h(P_i)=Q_i$ for $i=1,2$ and  
\[
[ pr_U^*(h_U) ]_2==[ pr_V^*(h_V) ]_2 \quad \text{in } \mathcal{S}^2(Q^{\psi})(W),
\]
with $pr_U, pr_V$ denoting the natural projections from $W=U \times_{\mathbb{A}^1_k} V$ onto $U$ and $V$, respectively.  

Assume that $V$ is irreducible (otherwise, the argument applies separately to each irreducible component). We must show that $[Q_1]_2=[Q_2]_2$ in $\mathcal{S}^2(Q^{\psi})(k)$.  

First, suppose that $\psi_{k(W)}$ is isotropic. Then  
\[
\mathcal{S}(Q^{\psi})(k(W))=\mathcal{S}^2(Q^{\psi})(k(W)).
\]
This implies  
\[
[ pr_U^*(h_U) ]_1=[ pr_V^*(h_V) ]_1 \quad \text{in } \mathcal{S}(Q^{\psi})(k(W)).
\]
Applying Lemma \ref{lemma on lines of brs}, we conclude that $[Q_1]_2=[Q_2]_2$ in $\mathcal{S}^2(Q^{\psi})(k)$, completing the proof in this case.

Now assume that $\psi_{k(W)}$ is anisotropic. Since  
\[
[pr_U^*(h_U)]_2=[pr_V^*(h_V)]_2 \quad \text{in } \mathcal{S}^2(Q^{\psi})(W),
\]  
it follows from Theorem~\ref{iteration 2 for anisotropic} that 
\[
pr_U^*(h_U) \sim_{\phi_{k(W)}} pr_V^*(h_V).
\]
Consequently, there exists a chain of good $k(W)$-rational curves 
\[
F^i: \mathbb{A}^1_{k(W)} \dashrightarrow Q^{\psi}
\]
connecting $k(W)$ rational points $pr_U^*(h_U)$ and $pr_V^*(h_V)$.  
Assume they are connected by a single such map $F$.  
Then
\[
F: \mathbb{A}^1_{k(W)} \dashrightarrow Q^{\psi}
\]
satisfies the following:
\begin{enumerate}
    \item $F(t_1)=pr_U^*(h_U)$ and $F(t_2)=pr_V^*(h_V)$ for some $t_1, t_2 \in k$;
    \item $F^*(x_1-1) \in f \< D(\phi_{k(W)(t)}) \>$ for some $f \in k(W)^*$.
\end{enumerate}

\item \textbf{Claim 1:} There exists $c\in k^*$ such that $F^*(x_1-1)\in c\< D(\phi_{k(W)(t)})\>$.
\begin{proof}[Proof of Claim 1]
Let $F_1(t)\in k(W)(t)$ denote $F^*(x_1)$. By(2), we have $F_1(t)-1\neq 0$. Moreover, we can assume $F_1(a)-1 \neq 0$ for any $a\in k$ where $F$ is defined. Indeed, suppose $F_1(a)-1=0$, then $F_1(a) + 1=2$. Therefore, 
\[F_1(t)+1 \in 2\< D(\phi_{k(W)(t)})\>.\] Moreover, since  
\[
F_1^2(t)-1 \in \< D(\phi_{k(W)(t)})\>
\quad \text{and} \quad F_1(t)-1\neq 0, 
\]  
it follows that  
\[
F_1(t)-1 \in 2\< D(\phi_{k(W)(t)})\>,
\]  
so we may take $c=2$ and we are done. 

Hence, we assume both $F_1(t_1),F_1(t_2) \neq 1$. In this case, we have  
\[
F_1(t_1)-1, F_1(t_2)-1 \in f\< D(\phi_{k(W)})\>.
\] 
Since $F_1(t_1)-1=p^*(h_U^*(x_1-1))$, the proof of Claim 1 will be complete if $h_U$ is a good $k$-rational curve. 

The proof that $h_U$ is a good rational curve follows similarly to the proof of Claim 1 in Theorem \ref{iteration 2 for isotropic case}. We need to show that  
\[
h_U^*(x_1-1) \in c\< D(\phi_{k(t)})\>  
\]  
for some $c \in k^*$.  

Suppose not. Then, by QVT there exists a closed point $P$ of $\mathbb{A}^1_k$ such that $v_P(h_U^*(x_1-1))$ is odd and $\phi_{k(P)}$ is anisotropic. By Lemma \ref{curve lemma for anisotropic case}, $h_U$ is not defined at $P$. Hence, there exists $P' \in V$ such that $p(P')=P$ and $k(P')=k(P)$.  Since  
\[
F_1(t_1)-1=p^*(h_U^*(x_1-1))  
\quad \text{and} \quad  
F_1(t_2)-1=h_V^*(x_1-1),  
\]  
we obtain  
\[
p^*(h_U^*(x_1-1)) \cdot h_V^*(x_1-1) \in \< D(\phi_{k(W)})\>.
\]  
Moreover, since $p$ is unramified, we have  
\[
v_{P'}(p^*(h_U^*(x_1-1)))=v_P(h_U^*(x_1-1)),  
\]  
which is odd.  Since $\phi_{k(P')}$ is anisotropic, it follows that  
\[
v_{P'}(p^*(h_U^*(x_1-1)) \cdot h_V^*(x_1-1))  
\]  
is even. This implies that $v_{P'}(h_V^*(x_1-1))$ is odd, leading to $h_V$ not being defined at $P'$ by Lemma \ref{curve lemma for anisotropic case}. This contradicts the assumption that $h_V$ is defined on $V$.
\end{proof}

\item \textbf{Claim 2:} If $T$ is a smooth  surface birational to $\mathbb A^1_k\times V$, $G$ is the rational map $G: T \dashrightarrow Q^{\psi}$ induced by $F$ and $E$ is a smooth $k$-rational curve embedded in $T$, then $G|_{E}$ is either constant or it is a good rational curve.

\begin{proof}[Proof of Claim 2]
The rational map $F: \mathbb{A}^1_{k(W)} \dashrightarrow Q^{\psi}$ induces a rational map  
\[
G: T \dashrightarrow Q^{\psi}.
\]  
Since $T$ is a smooth surface, the extension $\bar{G}: T \dashrightarrow \bar{Q^{\psi}}$ is defined everywhere except at finitely many points. As $\bar{Q^{\psi}} \setminus Q^{\psi}$ has no $k$-rational points, the images of $k$-rational points must lie in $Q^{\psi}$. Since $E(k)\neq\emptyset$, $G$ is defined on an open subset of $E$.  Thus,  
\[
G^*(y_i)=F_i \in \mathcal{O}_{E, T},
\]  
where $F^*(x_i)=F_i$. Therefore, the restriction of $G$ to $E$ is given by the rational map  
\[
G|_E: E \dashrightarrow Q^{\psi}, \quad (\bar{F}_1, \dots, \bar{F}_n),
\]  
where $\bar{F}_i$ are the images of $F_i$ under the natural map  $\mathcal{O}_{E,T} \to \mathcal{O}_{E,T} / m_E$. Since we have already established that  
\[
F_1(t)-1 \in c\< D(\phi_{k(W)(t)})\> \quad \text{for some } c \in k^*,
\]  
it follows that  
\[
\bar{F}_1-1 \in c\< D(\phi_{k(E)})\> \quad \text{or} \quad \bar{F}_1-1=0.
\]  
Since the closed subscheme  $V(x_1-1)\subset Q^{\psi}$ has only one $k$-rational point, the proof is complete.
\end{proof}
\begin{notation}
    If $T$ is a smooth surface as above, $G:T\dashrightarrow Q^{\psi}$ is the rational map induced by $F$,  and $C$ is any smooth curve in $T$, then $G_C$ will denote the morphism $C\rightarrow \bar Q^{\psi}$ induced by $G|_C$.
\end{notation}
Let $G$ denote the rational map induced by $F$ for the surface $\mathbb A^1_k\times V$, that is \[G:\mathbb A^1_k\times V\dashrightarrow  Q^{\psi}\] such that $G(t_1)|_W=pr_U^*(h_U)$ and $G(t_2)|_W=h_V|_W$. Since $h_U$ is a good rational curve and  must be defined at all $k$-rational points of $\mathbb A^1_k$, from Theorem \ref{Almost A1} we have
\[[ h_U(P)]=[h_U(P')] \text{ in } \mathcal S^2(Q^{\psi})(k) \quad \forall P,P'\in \mathbb A^1_k(k).\]

Therefore, it is sufficient to show the following claim:

\item \textbf{Claim 3:} For any $P\in V(k)$,  $[ h_V(P)]_2=[h_U(p(P))]_2 \text{ in } \mathcal S^2(Q^{\psi})(k)$.

\begin{proof}[Proof of Claim 3]
Let $P\in V(k)$ be fixed. Consider the following curves in $\mathbb A^1_k\times V$
\begin{enumerate}
    \item  Let $M,N$ denote the copies of $V$ in $\mathbb A^1_k\times V$ corresponding to $\{t_1\}\times V$ and $\{t_2\}\times V$ respectively. Then $G_M$ is induced by $ h_U\circ p$ and $G_N=h_V$.
    \item  Let $L$ denote the line $\mathbb A^1_k\times P$.
\end{enumerate}
It is sufficient to show that $G_M(P)$ and $G_N(P)$ lie on a connected chain of good rational curves because
\[G_M(P)=\bar h_U(p(P))\quad\text{and}\quad G_N(P)=h_V(P).\]
 First, suppose $G$ is defined at $(t_1,P)$ and $(t_2,P)$, then \[G_M(P)=G(t_1,P)=G_L(t_1), \quad G_N(P)=G(t_2,P)=G_L(t_2). \]  From Claim 2, $G_L$ is a good rational curve and hence we are done. Now assume $G$ is not defined at either $(t_1,P)$ or $(t_2,P)$. We will resolve the indeterminacy of $G$ to find the required chain of rational curves. 
 
 Let $\pi:T\rightarrow \mathbb A^1\times V$ be the morphism obtained after successively blowing up points such that the indeterminacy of $G$ is resolved. 

\[
\begin{tikzcd}
T\arrow[dr, rightarrow, " \tilde G"] \arrow[d, "\pi"]  \\
\mathbb{A}^1 \times V \arrow[r, dashrightarrow, "G", swap] & \bar Q^{\psi}
\end{tikzcd}
\]

Let $L',M',$ and $N'$ be the strict transforms of $L,M$ and $N$ respectively. Let $m\in M'(k)$ such that $\pi(m)=(t_1,P)$. Let $n\in N'(k)$ such that $\pi(n)=(t_2,P)$. Since $(t_1,P)$ and $(t_2,P)$ lie on the line $L$, we can use Theorem \ref{AM} to conclude that there exist smooth rational curves in $T$(which are isomorphic to $\mathbb A^1_k$) $E_1,\dots E_r$ such that $m$ and $n$ can be connected by a chain of rational curves $L',E_1,\dots, E_r$. Therefore, $\tilde G(m)$ and $\tilde G(n)$ are contained in the chain of rational curves comprising $\tilde G_{L'}$ and $\tilde G_{E_i}$. Since, 
\[\tilde G(m)=\tilde G_{M'}(P)=G_M(P),\quad \tilde G(n)=\tilde G_{N'}(P)=G_N(P)\]
and $\tilde G_{L'}$ and $\tilde G_{E_i}$ are either constant or good rational curves by Claim 2, we are done.
\end{proof}
Hence, the proof is complete.
\end{proof}
\begin{proof}[Proof of Theorem \ref{main theorem for Q}(2)]
By Theorem \ref{iteration 3 for anisotropic case}, \cite[Theorem 1.1]{BRS} and Lemma \ref{lemma on stabilisation}, we have for any $F/k$,
\[\mathbb \pi_0^{\mathbb A^1}(Q^{\psi})(F)=\mathcal L(Q^{\psi})(F)=\mathcal S^2(Q^{\psi})(F).\qedhere\]    
\end{proof}

\section{\texorpdfstring{Characterization of $\mathbb A^1$-connected quadrics}{Characterization of A1 connected quadrics}}
It is well known that a smooth projective quadric $Q$ over a field $k$ is rational if and only if it has a $k$-rational point, i.e., $Q(k) \neq \emptyset$. Consequently, the $\mathbb{A}^1$-connectedness of smooth projective quadrics is completely determined by the existence of a rational point over the base field (see Theorem 3, \cite{AM}).  In this section, we extend this characterization to affine quadrics. Specifically, we describe the $\mathbb{A}^1$-connectedness of affine quadrics in terms of classical invariants of quadratic forms.

\begin{theorem}
    Suppose $X=Q^{\psi}$ with where $\psi$ is a quadratic form in $n\geq 3$ variables. Let $\varphi := \psi \perp \< -1 \>$ be the quadratic form corresponding to $\bar{X}$. Then $X$ is $\mathbb{A}^1$-connected if and only if one of the following conditions holds:
    \begin{enumerate}
        \item $i_0(\varphi) \geq 2$.
        \item $i_0(\varphi)=1$, $i_0(\psi)=0$, and $i_1(\psi) \geq 2$.
    \end{enumerate}
\end{theorem}

\begin{proof}
If $i_0(\varphi) \geq 1$, then $1\in D(\psi)$ and there exists a quadratic form $\phi$ in $n-1$ variables such that  
    \[
       \psi=\<1\>\perp\<-\phi\>, \quad \text{and } \varphi=\< 1, -1 \> \perp -\phi.
    \]
    
($\impliedby$)  By Theorem \ref{stable}, it suffices to show that $\pi_0^{\mathbb{A}^1}(X)(F)=*$ for all $F/k$ whenever (1) or (2) holds. Suppose condition (1) holds,  then $\phi$ is isotropic, and the result follows directly from Lemma \ref{lemma on phi isotropic case}.  

Now, suppose condition (2) holds. In this case, $\psi$ is anisotropic. Since $i_1(\psi) \geq 2$, Theorem \ref{alternative defn of i1} ensures that  $i_0(\psi_F) \geq 2$ whenever $\psi_F$ is isotropic. Since $-\phi$ is a codimension-1 subform of $\psi$,  $i_1(\psi) \geq 2$ implies that $\psi_F$ is isotropic if and only if $\phi_F$ is isotropic.  

This guarantees that every rational curve in $X$ is a good rational curve. Since $\bar{X}$ is $\mathbb{A}^1$-connected, we have $\mathcal{S}(\bar X)(F)=*$. Therefore, any two $F$-points of $X$ can be connected by a chain of good rational curves,  which, together with, Theorem \ref{iteration 2 for anisotropic} implies that $\mathcal{S}^2(Q^{\psi})(F)=*$ for all $F/k$, completing the proof.\\  

($\implies$) Suppose that $X$ is $\mathbb{A}^1$-connected. This implies that $X(k) \neq \varnothing$, which in turn ensures that $i_0(\varphi) \geq 1$. Consequently, $
\varphi=\< 1, -1 \> \perp -\phi.$ It suffices to show that for all cases other than (1) and (2), there exists a field extension $F/k$ such that $\psi_F$ is isotropic and $\mathcal{S}(X)(F) \neq *$. Indeed, since $\mathcal{S}(X)(F)=\pi_0^{\mathbb{A}^1}(X)(F)$ whenever $\psi_F$ is isotropic, we will be done. The remaining cases are as follows:

\begin{case}[$i_0(\varphi)=1$ and $i_0(\psi)=1$]
    In this case, $\psi$ is isotropic, but $\phi$ is anisotropic. By QVT, we have:
    \[
    t \notin \< D(\phi_{k(t)}) \>,
    \]
    which implies that $\mathcal{S}(X)(k(t)) \neq *$.
\end{case}

\begin{case}[$i_0(\varphi)=1$, $i_0(\psi)=0$, and $i_1(\psi)=1$]
 By Theorem \ref{Witt index of subform}, $\phi_{k(\psi)}$ is anisotropic. Therefore, by the QVT, we have:
    \[
    t \notin \< D(\phi_{k(\psi)(t)}) \>,
    \]
    which implies that $\mathcal{S}(X)(k(\psi)(t)) \neq *$.
    
\end{case}
In both cases, we have shown that $X$ is not $\mathbb{A}^1$-connected, which completes the proof.
\end{proof}

\begin{examples}
     The above theorem can be applied to produce numerous examples of affine quadrics that are either $\mathbb{A}^1$-connected or not $\mathbb A^1$-connected.

    \begin{enumerate}
    \item If $k$ is quadratically closed, then for any regular quadratic form $\varphi$ over $k$, we have
    \[
    i_0(\varphi)=\left\lfloor \frac{\dim \varphi}{2} \right\rfloor.
    \]
    Hence, any affine quadric of dimension $n > 0$ over such a field is either $\mathbb{A}^1$-connected or isomorphic to $\mathbb{G}_m \times \mathbb{A}^{n-1}_k$.

    \item Suppose $\varphi$ is a regular anisotropic form of dimension $2^n + 1$. Then $Q^{\varphi}$ is not $\mathbb{A}^1$-connected. Indeed, in this case, $i_0(\varphi)=0$ and $i_1(\varphi)=1$ (see Corollary 1 in \cite{Hoffmann}). In particular, the real quadrics defined by
    \[
    x_1^2 + \dots + x_{2^n+1}^2=1
    \]
    are not $\mathbb{A}^1$-connected.

    \item Suppose $\varphi=\< 1, 1, \dots, 1 \>$ is a quadratic form of dimension $2^n + \ell$, where $2 \leq \ell \leq 2^n$, and assume that $\varphi_k$ is anisotropic. Then the associated quadric $Q^{\varphi}$ is $\mathbb{A}^1$-connected. Indeed, we have $i_0(\varphi)=0$, and by Corollary 1 in \cite{Hoffmann}, $i_1(\varphi) \leq \ell$. It is well known that a Pfister form is either anisotropic or totally isotropic (see Corollary 6.3 in \cite{EKM}). Since $\varphi$ is a subform of the $2^{n+1}$-dimensional Pfister form $\< 1, \dots, 1 \>$, a standard dimension-counting argument shows that $i_1(\varphi) \geq \ell$, and hence $i_1(\varphi)=\ell \geq 2$. Thus, $Q^\varphi$ is $\mathbb{A}^1$-connected. Similar arguments apply to any Pfister neighbor.

    \end{enumerate}
\end{examples}


\end{document}